\documentclass[11pt,twoside]{article}

\usepackage{geometry}
\usepackage[utf8]{inputenc}
\usepackage[english]{babel}
\usepackage{amssymb,amsmath,amsthm} %maths
\usepackage{mathrsfs}
\usepackage{mathtools}
\usepackage{empheq}
\usepackage{color, graphicx} %used for font color
\usepackage{float, subcaption} 
\usepackage{epsfig, circuitikz}

\usepackage{bm}
\usepackage{bbm}

\usepackage{cite}

\usepackage{fancyhdr}

\numberwithin{equation}{section}
\newcommand{\bg}[1]{{\boldsymbol{#1}}} % Bolding greek symbols
\newcommand{\mc}[1]{{\mathcal{#1}}} %\mathcal symbols
\newcommand{\norm}[1]{{\left\| #1 \right\|}}
\newcommand{\p}[1]{{\left( #1 \right)}}
\newcommand{\set}[1]{{\left\{ #1 \right\}}}
\newcommand{\br}[1]{{\left[ #1 \right]}}
\newcommand{\abs}[1]{{\left| #1 \right|}}

\def\dt{{\Delta t}}
\def\dx{{\Delta x}}
\def\jph{{j+\frac{1}{2}}}
\def\jmh{{j-\frac{1}{2}}}

\def\eps{\varepsilon}

\theoremstyle{plain}                    % use "default" font
\newtheorem{thm}{Theorem}[section]
\newtheorem{lemma}[thm]{Lemma}

\newtheorem{rmk}[thm]{Remark}

\newcommand*\xbar[1]{%
  \hbox{%
    \vbox{%
      \hrule height 0.5pt % The actual bar
      \kern0.4ex%         % Distance between bar and symbol
      \hbox{%
        \kern-0.05em%      % Shortening on the left side
        \ensuremath{#1}%
        \kern-0.00em%      % Shortening on the right side
      }%
    }%
  }%
}

\allowdisplaybreaks[1]

\graphicspath{{Figures/}}

\title{ An Asymptotic-Preserving Scheme for Isentropic Flow in Pipe Networks}

\author{Michael Redle\thanks{Chair of
Applied and Computational Mathematics, RWTH Aachen University, 52062 Aachen, Germany;
{\tt redle@acom.rwth-aachen.de}} \ and Michael Herty,\thanks{Chair of Numrical Analysis, Institute for Applied Mathematics (IGPM), RWTH University, 52062 Aachen, Germany;
{\tt herty@igpm.rwth-aachen.de} %{\color{red}[will put grant info in acknowledgements] }
} }
\date{\today}

\begin{document}
\maketitle
\begin{abstract}
    {
We consider the simulation of isentropic flow in pipelines and pipe networks. 
Standard operating conditions in pipe networks suggest an emphasis to simulate low Mach and high friction regimes % of the underlying isentropic Euler equations with a nonlinear friction source. 
-- however, the system is stiff in these regimes and conventional explicit approximation techniques prove quite costly and often impractical.
To combat these inefficiencies, we develop a novel asymptotic-preserving scheme that is uniformly consistent and stable for all Mach regimes. 
The proposed method for a single pipeline follows the flux splitting suggested in 
[Haack et al., Commun. Comput. Phys., 12 (2012), pp. 955--980],
in which the flux is separated into stiff and non-stiff portions then discretized in time using an implicit-explicit approach.
The non-stiff part is advanced in time by an explicit hyperbolic solver; we opt for the second-order central-upwind finite volume scheme. 
The stiff portion is advanced in time implicitly using an approach based on Rosenbrock-type Runge-Kutta methods, which ultimately reduces this implicit stage to a discretization of a linear elliptic equation. 

To extend to full pipe networks, the scheme on a single pipeline is paired with coupling conditions defined at pipe-to-pipe intersections to ensure a mathematically well-posed problem.
We show that the coupling conditions remain well-posed in the low Mach/high friction limit -- which, when used to define the ghost cells of each pipeline, results in a method that is accurate across these intersections in all regimes.
The proposed method is tested on several numerical examples and produces accurate, non-oscillatory results with run times independent of the Mach number.
}
\end{abstract}

\bigskip


%=========================================
\section{Introduction}\label{sec1}

This paper focuses on the development of a novel numerical method for gas flow in pipelines and pipe networks. 
To describe the gas transport, we use the isothermal/isentropic Euler equations with a source to account for the friction along the pipe walls:

\begin{equation}
    \begin{aligned}
        & \rho_t + \p{\rho u}_x  = 0, \\
        &\p{\rho u}_t + \p{\rho u^2 +p}_x = -\frac{ \kappa}{2D} \rho u \abs{u},
    \end{aligned}
    \label{eq:euler}
\end{equation}
where $x$ is the (one-dimensional) spatial variable, $t$ denotes time, $\rho$ is the fluid density, $u$ denotes the fluid velocity, $p = \rho^\gamma$ is the pressure under the isentropic assumption, in which $\gamma$ is the ratio of specific heats, and $\kappa$ denotes the Fanning friction coefficient. 
In the one-dimensional approach, which is quite common when modeling gas flows in pipes due to the ratio of pipe length $L$ and cross section $D$; see e.g. \cite{Osiadacz1987Simulation,bressan2014flowsnetwor}; the unknowns $\rho$ and $u$ are averaged over the (assumed constant) cross section. 

The standard operating conditions hold the gas flow at moderate velocities; see e.g. \cite{Brouwer2011Gas} where they note a reference Mach number around 0.001; and thus we must account for the low Mach number or high friction regimes of \eqref{eq:euler}.
However, the limiting solution as this parameter goes to zero brings about a number of difficulties -- the most notable of which is that the underlying system becomes very stiff.
This makes designing efficient and accurate methods to approximate the solution quite challenging.
For example, conventional explicit schemes applied to system \eqref{eq:euler} would greatly over-resolve the solution in time, as the wave speeds of the system are proportional to the inverse of the reference Mach number -- further implying that the CFL stability restriction is proportional to this small parameter; see, e.g., \cite{GUILLARD1999behaviour,GUILLARD2004Behavior,RIEPER2010Dissipation}. 
Hence, explicit schemes may prove quite computationally expensive, especially in the low Mach/high friction limit.

One alternative to avoid this over-resolution issue are asymptotic-preserving (AP) schemes.
By the definition in \cite{Jin1995Runge,Jin1999Efficient}, a scheme is AP if the discretization of the continuous system remains consistent and stable regardless of the value taken for the underlying singular 
parameter (denoted by $\eps$ in this paper).
Furthermore, AP schemes allow for the space and time discretization to be independent of $\eps$; i.e., the CFL condition for stability does not depend on the small parameter in the system. 

Due to the potential speed up when discretizing with AP methods, they have attracted a lot of attention in the numerical and engineering community. 
AP schemes have been extensively studied for the kinetic equations; see, e.g., \cite{HU2017Handbook, Hu2018AP,Jin2000Uniformly,Jin2012AP,Jin2022AP,REN2014AP,ZHANG2016AP} and references therein. 
More recently, AP methods for the kinetic equations have been further extended to use AP Monte Carlo methods to reduce numerical diffusion effects \cite{FEI2023time,fei2024navier}, and to use AP neural networks to solve the underlying PDE system \cite{Jin2023APNN,Jin2024APNN}.
There has also been extensive development of AP methods in the low Mach limit of the isentropic Euler, compressible Euler, and Navier-Stokes systems in \cite{Arun2020AP,Arun2020APwG,AVGERINOS2019linearly,BISPEN2017AP,BOSCARINO2019high,Cordier2012AP,Degond2011AllSpeed,Dimarco2017Study,Haack2012AllSpeed,kanbar2024asymptotic,KLEIN1995Semi,Kucera2022Asymptotic,Noelle2014Weakly,samantaray2024asymptotic,Zeifang2020novel}, and in the low Froude limit of the shallow water equations in \cite{bispen2014imex,BOSCHERI2023All,BUSTO2022Staggered,DURAN2015AP,HUANG2022High,Kurganov2022WB,Liu2019AP,Vater2018Semi,XIE2024High,Zakerzadeh2017Finite}.

The number of studies significantly shrinks, though, when it comes to the consideration of nonlinear friction terms that remain in the limiting solution. 
When this nonlinear friction term, such as that in \eqref{eq:euler}, remains in the low Mach/Froude limit, it adds the difficulty that this source must be discretized implicitly in some manner. 
If discretized naively, this requires some a nonlinear solve, which would preferably be avoided. 
To our knowledge, only the work in \cite{BUSTO2022Staggered,Egger2023AP} consider some nonlinear friction term, in which only \cite{BUSTO2022Staggered} avoids a nonlinear solve via their proposed semi-implicit hybrid finite volume/finite element method. 

Furthermore, the number of studies of AP schemes in the extension to pipe networks is also quite limited. 
To our knowledge, there are only two such developments. 
In \cite{egger2022asymptotic}, they present an AP hybrid-DG method on networks for the linear convection-diffusion equation.
The work in \cite{Egger2023AP} proposes a finite element AP method applied to the barotropic Euler equations with a nonlinear friction term on pipe networks. 
However, the aformentioned method does opt for an implicit time discretization reliant on a nonlinear solver. 

In this paper, we propose an AP scheme for the isentropic Euler equations with a nonlinear friction source on pipe networks. 
The method on a single pipeline uses hyperbolic flux splitting proposed in \cite{Haack2012AllSpeed,Liu2019AP} to split the stiff and non-stiff parts of the system, which are then treated through an implicit-explicit approach. 
The non-stiff portion of the system is advanced in time explicitly and discretized in space using the second-order central-upwind (CU) finite volume scheme developed in \cite{Kurganov2001Semidiscrete,Kurganov2002Solution}.
The stiff part of the system is advanced in time implicitly using an approach related to Rosenbrock-type Runge-Kutta methods seen in, e.g., \cite{wanner1996solving,Zhong1996Additive}, 
which ultimately reduces this implicit stage to an elliptic equation that can then be solved linearly after discretizing in space via standard central-difference.
%which ultimately reduces this implicit stage to a linear solve of an elliptic equation after discretizing in space via standard central-difference.
The proposed scheme for a single pipeline is extended to pipe networks by defining coupling conditions, such as those in \cite{Banda2006Gas,Banda2006Coupling}, at pipe-to-pipe intersections to ensure a mathematically well-posed problem.
We show that the coupling conditions remain well-posed in the low Mach/high friction regimes, and use them to define the ghost cell values of each pipeline.
The resulting method is tested on several numerical examples of pipe networks and produces accurate and non-oscillatory results, in addition to running significantly faster than the analogous fully explicit scheme in the low Mach/high friction limit. 

%{\color{red}To extend to full pipe networks, the proposed single pipe scheme is paired with coupling conditions defined at pipe-to-pipe intersections to ensure a mathematically well-posed problem.
%We show that the coupling conditions, used to set the ghost cells of the single pipe discretizations, remain well-posed in the low Mach/high friction regimes, in turn allowing for a seamless coupling of the method across these intersections.
%The resulting method is tested on several numerical examples -- producing accurate, non-oscillatory results that run significantly faster in the low Mach/high friction limit in comparison to the analogous fully explicit scheme.}

This paper is organized as follows. 
In \S\ref{sec2}, we discuss the asymptotics of the isentropic Euler equations \eqref{eq:euler} on pipe networks and associated numerical difficulties.
We then develop an AP scheme for the underlying system in \S\ref{sec3}. 
We present the performance of the proposed scheme on three examples in \S\ref{sec4}, and make some concluding remarks in \S\ref{sec5}.

\section{Dimensional Analysis}\label{sec2}
To look at the asymptotics of \eqref{eq:euler}, we non-dimensionalize the system by introducing characteristic time $t_0$, characteristic density $\rho_0$, characteristic velocity $w_0$, and characteristic pressure $p_0$, along with using the pipe length $L$ as the characteristic length. 
Therefore, the dimensionless quantities for system \eqref{eq:euler} read 
\begin{equation*}
    \widehat{x} = \frac{x}{L}, \qquad 
    \widehat{t} = \frac{t}{t_0}, \qquad 
    \widehat{\rho} = \frac{\rho}{\rho_0}, \qquad 
    \widehat{u} = \frac{u}{w_0}, \qquad 
    \widehat{p} = \frac{p}{p_0}.
    %\label{eq:nondim}
\end{equation*} 
Taking $w_0 = L/t_0$, substituting these quantities into \eqref{eq:euler}, and dropping the hat notation for the sake of simplicity, we obtain the dimensionless isentropic Euler equations with a friction source term: 
\begin{equation*}
    \begin{aligned}
        & \rho_t + \p{\rho u}_x  = 0, \\
        &\p{\rho u}_t + \p{\rho u^2 + \frac{1}{\textrm{Ma}^2}p}_x = -\frac{\kappa}{2\delta} \rho u \abs{u},
    \end{aligned}
    %\label{eq:euler_nd}
\end{equation*}
where 
$$\textrm{Ma} = w_0\sqrt{\frac{\rho_0}{p_0}}, \qquad \delta = \frac{D}{L},$$
are the reference Mach number and the ratio of the cross section to the pipe length, respectively.
We then choose to take the reference Mach number $\rm{Ma} = \eps$, and follow the suggestions of \cite{Brouwer2011Gas, Egger2023AP, Egger2023Stability} in taking $\delta = \eps^2/C_\delta$, where $C_\delta$ is a constant, resulting in the parameterized system  
\begin{equation}
    \begin{aligned}
        & \rho_t + \p{\rho u}_x  = 0, \\
        &\p{\rho u}_t + \p{\rho u^2 + \frac{1}{\eps^2}p}_x = -\frac{C_\delta \kappa}{2\eps^2} \rho u \abs{u}, 
    \end{aligned}
    \label{eq:euler_eps}
\end{equation}
which can otherwise be written in the following vector form
\begin{equation}
    \bm U_t + \bm F\p{\bm U}_x = \bm S\p{\bm U}, \quad 
    \bm F\p{\bm U} = \begin{pmatrix}
        \rho u \\[4pt]
        \rho u^2 + p/\eps^2
    \end{pmatrix}, \quad 
    \bm S\p{\bm U} = \begin{pmatrix}
        0 \\[4pt]
        -\frac{C_\delta \kappa}{2\eps^2} \rho u \abs{u}
    \end{pmatrix},
    \label{eq:vecform}
\end{equation}
where $\bm U := (\rho, \rho u)^\top$, $\bm F(\bm U)$ denotes the flux, and $\bm S(\bm U)$ is the source due to friction along the pipe walls.

\begin{rmk}\label{rmk2.1}
    Note that in the limit $\eps \rightarrow 0$, it is clear that system \eqref{eq:euler_eps} approaches a state where the spatial derivative of pressure balances the friction source due to the pipe walls. 
    This follows a classical limit commonly seen in the modeling of pipelines and pipe networks; see, e.g., {\cite{Brouwer2011Gas}}.
\end{rmk}

%==========================================
\subsection{Continuous Extension to Pipe Networks}\label{sec2.2new}
To simulate an entire pipe network, we must of course consider appropriate boundary conditions at each pipe entrance and exit -- the most complicated of which are at pipe-to-pipe intersections.
These are obtained by defining some coupling conditions at the pipe-to-pipe intersections or junctions; see, e.g., \cite{Banda2006Gas,Ehrhardt2005Nonlinear}.
In this section, we briefly describe some suitable condition options that may be prescribed at pipe junctions.
The following coupling conditions presented are commonly used examples for the mathematical framework of pipeline networks; see, for example, \cite{Herty2008Sim,Herty2013Assessment,Herty2019Fast} and references therein. 

Let us denote the entrance and exit of pipe $k = 1, \ldots, K$ as $x_{\rm i}^{(k)}$ and $x_{\rm f}^{(k)}$, respectively, and the variables in pipe $k$ as $\bm U^{(k)} = (\rho^{(k)}, (\rho u)^{(k)})^\top$. 
Consider a single junction in which pipes with indices $1,\ldots, m$ denote the ingoing pipelines and those with indices $m+1,\ldots, K$ are the outgoing pipelines, as depicted in Figure \ref{fig:junction}. 
Then at the junction, one must have the coupling condition for:\\

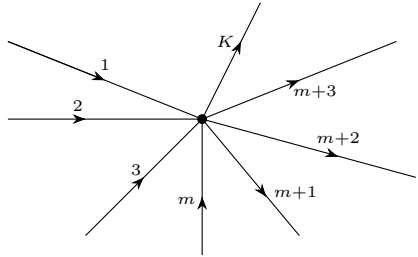
\begin{figure}[!ht]
\centering
\vspace*{-50pt}\resizebox{.4\textwidth}{!}{%
\begin{circuitikz}
\tikzstyle{every node}=[font=\tiny]
\draw (5.75,13.5) to[short] (5.75,13.5);
%\draw [->, >=Stealth] (9.25,12) -- (9.25,12.75)node[pos=0.5, fill=white]{Text};
\node at (9.25,10.25) [circ] {};
\node at (7.65,10.42) [fill=white] {2};
\node at (8,10.95) [fill=white] {1};
\node at (8.4,9.6) [fill=white] {3};
\node at (9.05,9.2) [fill=white] {$m$};
\node at (10.45,9.3) [fill=white] {$m$+1};
\node at (11,10) [fill=white] {$m$+2};
\node at (10.7,10.6) [fill=white] {$m$+3};
\node at (9.55,11.25) [fill=white] {$K$};
\draw [short] (9.25,10.25) -- (6.75,10.25);
\draw [short] (9.25,10.25) -- (6.75,11.25);
\draw [short] (9.25,10.25) -- (7.75,8.75);
\draw [short] (9.25,10.25) -- (9.25,8.5);
\draw [short] (9.25,10.25) -- (10,11.75);
\draw [short] (9.25,10.25) -- (11.75,11.25);
\draw [short] (9.25,10.25) -- (12,9.5);
\draw [short] (9.25,10.25) -- (10.5,8.75);
\draw [->, >=Stealth] (7.74,10.25) -- (7.75,10.25);
\draw [->, >=Stealth] (6.75,11.25) -- (8,10.75);
\draw [->, >=Stealth] (8.499,9.499) -- (8.5,9.5);
\draw [->, >=Stealth] (9.25,9.24) -- (9.25,9.25);
\draw [->, >=Stealth] (10.0742,9.251) -- (10.075,9.25);
\draw [->, >=Stealth] (10.99,9.7685) -- (11,9.765);
\draw [->, >=Stealth] (10.499,10.7496) -- (10.5,10.75);
\draw [->, >=Stealth] (9.74991,11.2498) -- (9.75,11.25);
\end{circuitikz}
}%
\caption{\sf Illustration of a pipe network junction with $m$ ingoing pipelines and $K-m$ outgoing pipelines.}
\label{fig:junction}
\end{figure}

\noindent Conservation of momentum:
\begin{equation}\label{eq:cons_ru}
    \sum_{k = 1}^m (\rho u)^{(k)}(x_{\rm f}^{(k)},t) = \sum_{\ell = m+1}^K (\rho u)^{(\ell)}(x_{\rm i}^{(\ell)},t),
\end{equation}
paired with one of the following:
\begin{itemize}
    \item[(a)] Equal pressures (see, e.g., \cite{Banda2006Gas,Banda2006Coupling}):
    \begin{equation}\label{eq:eq_p}
        p\p{\rho^{(k)}(x_{\rm f}^{(k)},t)} = p\p{\rho^{(\ell)}(x_{\rm i}^{(\ell)},t)} \qquad  \forall k = 1,\dots, m, \qquad  \ell = m+1,\dots,K; 
    \end{equation}

    \item[(b)] Equal momentums (see, e.g., \cite{Colombo2006Well}):
    \begin{equation}\label{eq:eq_ru}
        \begin{aligned}
        &\p{\rho^{(k)}(u^{(k)})^2}(x_{\rm f}^{(k)},t) + \frac{1}{\eps^2}p\p{\rho^{(k)}(x_{\rm f}^{(k)},t)}
        =\p{\rho^{(\ell)}(u^{(\ell)})^2}(x_{\rm i}^{(\ell)},t) + \frac{1}{\eps^2}p\p{\rho^{(\ell)}(x_{\rm i}^{(\ell)},t)}\\[6pt]
        &\hspace{7.5cm}\forall k = 1,\dots, m, \qquad  \ell = m+1,\dots,K;
        \end{aligned}
    \end{equation}

    \item[(c)] Geometry or flow-dependent pressure loss (see, e.g., \cite{Osiadacz1987Simulation,white2003fluid}):
    \begin{equation}\label{eq:p_loss}
        \frac{1}{\eps^2}p\p{\rho^{(k)}(x_{\rm f}^{(k)},t)} 
        = \frac{1}{\eps^2}p\p{\rho^{(\ell)}(x_{\rm i}^{(\ell)},t)} - h_{k,\ell}, \qquad \forall k = 1,\dots, m, \qquad  \ell = m+1,\dots,K,
    \end{equation}
    where $h_{k,\ell}$ denotes the pressure loss at the junction. %{\color{red}[Should this have $\eps^2$ piece?]}
    %\mh{ I would add in Section 2.2, the limiting boundary conditions (i.e. equal pressure). No numerics. }

\end{itemize}
Notice that the $\eps$-dependence remains in the coupling conditions, as it arose from the non-dimensionalization of system \eqref{eq:euler}. 

The isentropic Euler system \eqref{eq:euler_eps} paired with \eqref{eq:cons_ru} and one of \eqref{eq:eq_p}--\eqref{eq:eq_ru} has been proven a well-posed problem under the condition that the initial data (i) is not in a vacuum state ($\rho^{(k)}(x,t = 0) > 0\ \forall k$); (ii) has a flow direction that does not change ($u^{(k)}(x,t = 0) \geq 0\ \forall k$); and (iii) is under subsonic conditions ($u^{(k)}(x,t = 0) \leq c\ \forall k$) \cite{Banda2006Coupling,Banda2006Gas,Colombo2006Well}.
%While the isothermal Euler system \eqref{eq:euler} paired with \eqref{eq:cons_ru} and \eqref{eq:p_loss} is not yet known as a well-posed problem 
These coupling conditions are implemented into the boundary conditions of each pipe $k$; we present how this is done for the proposed scheme in \S\ref{sec3.5.1} following the rich literature in this field, e.g., \cite{godlewski2004numerintercoupl,godlewski2005, ambroso2008coupllagran,chalons2008,MR922200,holden1995,muller2015,bretti2006fast,borsche2016numer,banda2016numer}. 

%\mh{ Up to here, we only have the continuous description of the boundary conditions. Those could be stated also in Section 2.0. This would then also be more different from Alina's paper. 
%}

%=========================================
\subsection{The Low Mach/High Friction Limit}\label{sec2.2}
To investigate the asymptotic behavior of system \eqref{eq:euler_eps} inside each pipeline, we take the asymptotic expansion of  variables $\rho$ and $u$. 
Note that for simplicity, we removed the superscript $(k)$ introduced in the previous section for now, as the coupling conditions are only introduced on the boundaries. 
Thus, the associated asymptotic expansions of our unknowns are
\begin{equation*}
    \begin{aligned}
        \rho &= \rho^{(0)} + \eps^2 \rho^{(2)} + \cdots, \\
        u &= u^{(0)} + \eps^2 u^{(2)} + \cdots.
    \end{aligned}
    %\label{eq:A_expan}
\end{equation*}
Consequentially, we can obtain the asymptotic expansion for pressure $p = \rho^\gamma$ using Taylor series:
\begin{equation}
    p = \p{\rho^{(0)}}^\gamma + \eps^2 \gamma \p{\rho^{(0)} }^{\gamma-1}\rho^{(2)}+ \cdots.
    \label{eq:A_expan_pr}
\end{equation}
Note that the $\eps^1$ terms are skipped since there are no $\mathcal{O}(\eps^{-1})$ appearing in system \eqref{eq:euler_eps}.
Substituting the expansions of $\rho,\ u,\ p$ into system \eqref{eq:euler_eps} and collecting like powers of $\eps$, we obtain the asymptotic behavior of the isentropic Euler equations with a friction source:
\begin{equation}
    \begin{aligned}
        \mathcal{O}(\eps^{-2}):\quad & \br{\p{\rho^{(0)}}^\gamma}_x = -\frac{C_\delta \kappa}{2}\rho^{(0)} u^{(0)} \abs{u^{(0)}},\\[4pt]
        \mathcal{O}(1) :\quad & \rho^{(0)}_t + \p{\rho^{(0)}u^{(0)}}_x = 0, \\[2pt]
        &\p{\rho^{(0)}u^{(0)}}_t + \br{\rho^{(0)}\p{u^{(0)}}^2 + \gamma \p{\rho^{(0)}}^{\gamma -1} \rho^{(2)}}_x 
        \\
        &\hspace{2cm}= -\frac{C_\delta \kappa}{2} \p{\rho^{(2)} u^{(0)} \abs{u^{(0)}} + \rho^{(0)} u^{(2)} \abs{u^{(0)}} + \rho^{(0)} u^{(0)} \abs{u^{(2)}}}.
    \end{aligned}
    \label{eq:A_Behav}
\end{equation}
Here, note that the $\mathcal{O}(\eps^{-2})$ asymptotic is exactly that discussed in Remark \ref{rmk2.1}.

Similarly, we consider the $\eps \rightarrow 0$ case on the coupling conditions at the pipe intersections. 
It is clear that (i) the condition related to conservation of momentum will remain as it is $\eps$-independent; and (ii) regardless of your selection of coupling conditions \eqref{eq:eq_p}, \eqref{eq:eq_ru}, or \eqref{eq:p_loss}, all simplify to requiring equal pressures at the junction in the low Mach/high friction limit. Note that the different limiting equations have already been presented e.g. in \cite{Brouwer2011Gas}. The main aspect here is on their numerical treatment.

%\mh{ The equation for $\mathcal{O}(\eps^{-2})$ is the small epsilon limit. My question is, if you do the limit in \eqref{eq:IMEXrho},\eqref{eq:IMEXru} or \eqref{eq:discrete_rho},\eqref{eq:discrete_ru},
%what limiting scheme to you obtain? Is this consistent for $\Delta x,\Delta t\to 0$ for this equation? 
%}

%=======================================
\subsection{Numerical Difficulties in Low Mach/High Friction regimes}\label{sec2.1}

By computing the eigenvalues of the Jacobian $\partial \bm F/ \partial \bm U$, one can find that the wave speeds of system \eqref{eq:vecform} are
$$ \set{ u \pm \frac{1}{\eps}\sqrt{p'(\rho)}}. $$
In turn, this implies that if one was to solve system \eqref{eq:euler_eps} using some standard explicit method with a uniform mesh spatial discretization with cell size $\dx$, the corresponding time-step restriction due to the CFL condition would be
\begin{equation}
    \dt_{\rm{ex}} \leq \nu \frac{\dx}{\displaystyle{\max_x}\set{\abs{u} + \frac{1}{\eps}\sqrt{p'(\rho)}}} = {\mathcal{O}(\eps \dx)},
    \label{eq:CFL_ex}
\end{equation}
% \mh{My point here is the following: The previous condition is
% only relevant for explicit discrete schemes. Those we have not introduced yet and therefore would have that in 3.4, as in equation \eqref{eq:CFL_ap} }
where $0 < \nu \leq 1$ denotes the CFL number. 
On top of this, explicit schemes typically have numerical diffusion of ${\mathcal{O}\p{(\dx)^p/\eps}}$ \cite{Haack2012AllSpeed}, where $p$ is the order of the method, further implying that one would need to select $\dx = {\mathcal{O}(\eps^{1/p})}$ to combat excessive numerical diffusion in the results. 
Therefore, to obtain unsmeared results, the time-step restriction for explicit schemes would need to be $\dt_{\rm{ex}} = \mathcal{O}(\eps^{1+1/p})$. 
In other words, explicit schemes are quite inefficient in the low Mach and high friction regimes due to the significant computational cost.

An alternative to avoid the heavy time-step restriction seen in \eqref{eq:CFL_ex} would instead be to advance in time using an implicit method. 
However, this also could prove quite costly, as the nonlinearity of system \eqref{eq:euler_eps} implies a dependence on some nonlinear iterative solver of an $N\times N$ system of equations, where $N$ is the number of cells in the spatial discretization.

Thus, we want a scheme that removes this $\eps$-dependence in the time-step restriction, converges to the asymptotics in \eqref{eq:A_Behav}, and maintains the correct coupling conditions in the $\eps \rightarrow 0$ limit. 
The scheme proposed in the following section aims to meet these desired properties on the discrete level.

%{\color{red}CONTINUE HERE}

%=========================================
\section{Asymptotic-Preserving Scheme}\label{sec3}

To form an AP scheme for the system \eqref{eq:euler_eps} on a single pipeline, we follow the hyperbolic flux splitting idea from \cite{Haack2012AllSpeed,Liu2019AP}. 
To do so, we separate the slow and fast dynamics into two parts, resulting in the corresponding split system:
\begin{equation}
    \begin{aligned}
        & \rho_t + \alpha\p{\rho u}_x + \p{1-\alpha} \p{\rho u}_x = 0, \\[4pt]
        &\p{\rho u}_t + \p{\rho u^2 + \frac{p - a(t) \rho}{\eps^2}}_x + \frac{a(t)}{\eps^2}\rho_x = -\frac{C_\delta f}{2\eps^2} \rho u \abs{u}.
    \end{aligned}
    \label{eq:split}
\end{equation}
Equivalently, the split form \eqref{eq:split} can be written into an updated vector form, which reads 
\begin{equation}
    \bm U_t + \widetilde{\bm F}\p{\bm U}_x + \widehat{\bm F}\p{\bm U}_x = \bm S\p{\bm U},
    \label{eq:newvecform}
\end{equation} 
where 
\begin{equation}
    \widetilde{\bm F}\p{\bm U} = \begin{pmatrix}
        \alpha \rho u \\[4pt]
        \displaystyle{\rho u^2 + \frac{p - a(t) \rho}{\eps^2}}
    \end{pmatrix}, \qquad 
    \widehat{\bm F}\p{\bm U} = \begin{pmatrix}
        \p{1-\alpha}\rho u \\[4pt]
        \displaystyle{\frac{ a(t) \rho}{\eps^2}}
    \end{pmatrix}, 
    \label{eq:splitflux}
\end{equation}
are the slow (non-stiff) and fast (stiff) dynamics parts of the fluxes, respectively, and $\bm S(\bm U)$ is defined in \eqref{eq:vecform}. 

To guarantee the non-stiff subsystem $\bm U_t + \widetilde{\bm F}\p{\bm U}_x = \bm 0$ is indeed non-stiff and hyperbolic, $\alpha$ and $a(t)$ must be chosen appropriately to remove the $1/\eps^2$ dependence on the wave speeds. 
Computed by finding the eigenvalues of the Jacobian $\partial \widetilde{\bm F}/ \partial \bm U$, the wave speeds of this subsystem are
\begin{equation}
     \set{u \pm \sqrt{\p{1-\alpha}u^2 + \frac{\alpha\br{p'(\rho)-a(t)}}{\eps^2}}}.
     \label{eq:eigen}
\end{equation}
Thus, to ensure the eigenvalues of the non-stiff subsystem are real and are $\mc O (1)$ instead of $\mc O (\eps^{-2})$ as in the original system \eqref{eq:euler_eps}, we take 
\begin{equation}
    \alpha = \eps^b \qquad \textrm{and} \qquad a(t) = \min_{x} p'(\rho),
    \label{eq:alpha_a}
\end{equation}
and $b \geq 2$. 
In turn, the new wave speeds in \eqref{eq:eigen} now avoid the dissipation and time-step issues seen in original system \eqref{eq:euler_eps}, allowing the slow dynamics to be discretized using any appropriate hyperbolic solver. 
We describe the hyberbolic solver we use in \S\ref{sec3.2}.

\begin{rmk}\label{rmk3.1}
    Note that, in comparison to the work of \cite{Haack2012AllSpeed,Liu2019AP}, we have less freedom in the selection of this parameter $b$ arising in \eqref{eq:alpha_a}. 
    This is due to the fact that the asymptotic behavior shown in \eqref{eq:A_Behav} allows for a non-constant $\rho^{(0)}$ when $\eps \rightarrow 0$, meaning $\br{p'(\rho)-a(t)} $ is  $\mc O(1).$
\end{rmk}

% The fast dynamics, which read 
% \begin{equation*}
%     \begin{aligned}
%         & \rho_t + \p{1-\alpha} \p{\rho u}_x = 0, \\[4pt]
%         &\p{\rho u}_t + \frac{a(t)}{\eps^2}\rho_x = -\frac{C_\delta f}{2\eps^2} \rho u \abs{u}, 
%     \end{aligned}
% \end{equation*}

%=======================================
\subsection{Time Discretization}\label{sec3.1}
To discretize the split system \eqref{eq:newvecform}--\eqref{eq:splitflux} in a way that relaxes the time-step stability restriction, we use the implicit-explicit (IMEX) approach; that is, we approximate the non-stiff flux terms $\widetilde{\bm F}(\bm U)_x$ explicitly, and use an implicit approximation for the stiff flux terms $\widehat{\bm F}(\bm U)_x$ and the friction source $\bm S(\bm U).$
Since the source term $\bm S (\bm U)$ is nonlinear, we opt for a time discretization related to Rosenbrock-type Runge-Kutta methods, which will still allow the system to be solved without the need of nonlinear solvers; see, e.g., \cite{wanner1996solving,Zhong1996Additive} and references therein. 
To this end, we use the following first-order IMEX time discretization:
\begin{align}
    &\frac{\rho^{n+1} - \rho^n}{\dt} + \alpha \p{\rho u}^n_x + \p{1-\alpha} \p{\rho u}^{n+1}_x = 0, 
    \label{eq:IMEXrho}\\[4pt]
    &\frac{\p{\rho u}^{n+1} - \p{\rho u}^n}{\dt} + \p{\rho u^2 + \frac{p -a^n \rho}{\eps^2}}^n_x + \frac{a^n}{\eps^2}\rho^{n+1}_x = -\frac{C_\delta \kappa}{2\eps^2} \abs{u^n}\p{\rho u}^{n+1}. 
    \label{eq:IMEXru}
\end{align}
The influence of the Rosenbrock-type method appears in the discretization of the source of \eqref{eq:IMEXru}, in which only the $\rho u$ term is evaluated at time $t^{n+1}$.
%, thus allowing a linear solve in the computation of $(\rho u)^{n+1}$. 
For simplicity and notation purposes, let us define 
\begin{equation*}%\label{eq:R}
    \bm R^n := (R^{\rho,n}, R^{\rho u, n})^\top = -\widetilde{\bm F}(\bm U)_x. 
\end{equation*}
One can then solve equations \eqref{eq:IMEXrho}--\eqref{eq:IMEXru} for $\rho^{n+1}$ and $(\rho u)^{n+1}$, respectively, to obtain
\begin{align}
    \rho^{n+1} &= \rho^n + \dt R^{\rho,n} - \dt (1-\alpha)\p{\rho u}^{n+1}_x, 
    \label{eq:semi_rho}\\
    \p{\rho u}^{n+1} &= \frac{1}{\Psi^n(x)}\br{\p{\rho u}^n + \dt R^{\rho u, n} - \frac{a^n \dt}{\eps^2}\rho^{n+1}_x},
    \label{eq:semi_ru}
\end{align}
where 
\begin{equation}\label{eq:psi}
    \Psi^n(x) := 1 +  \dt \cdot \frac{C_\delta \kappa}{2\eps^2} \abs{u^n}
\end{equation}
is always greater than 1.
Following then that of \cite{Degond2011AllSpeed}, we then differentiate  \eqref{eq:semi_ru} with respect to $x$ and substitute this into \eqref{eq:semi_rho} to obtain an elliptic equation for $\rho^{n+1}:$
\begin{equation}
    \rho^{n+1} - \frac{\dt^2}{\eps^2}a^n(1-\alpha)\p{\frac{\rho^{n+1}_x}{\Psi^n}}_x 
    = \rho^n + \dt R^{\rho,n} -\dt (1-\alpha)\br{\frac{\p{\rho u}^n + \dt R^{\rho u,n}}{\Psi^n}}_x.
    \label{eq:elliptic_rho}
\end{equation}
Assuming all values at time $t^n$ are known, this implies that the choice of the Rosenbrock-type method in \eqref{eq:IMEXru} results in just a linear equation for the unknown $\rho^{n+1}$.
Furthermore, the time discretization in \eqref{eq:IMEXru} in combination with an appropriate spatial discretization, such as that further detailed in \S\ref{sec3.2} and \S\ref{sec3.3}, will result in \eqref{eq:elliptic_rho} being a linear system that can be solved for $\rho^{n+1}$. 
This solution of $\rho^{n+1}$ is then substituted into \eqref{eq:semi_ru} to obtain $(\rho u)^{n+1}$.

% With an appropriate spatial discretization, further discussed in \S\ref{sec3.2} and \S\ref{sec3.3}, \eqref{eq:elliptic_rho} is just a linear system that can be solved for $\rho^{n+1}$, whose solution can be substituted into \eqref{eq:semi_ru} to obtain $(\rho u)^{n+1}$.

%=================================
\subsection{Spatial Discretization of Non-Stiff Flux Terms}\label{sec3.2}
We use the Central-Upwind (CU) finite volume scheme to discretize the non-stiff flux terms $\widetilde{\bm F}(\bm U)$.
For each pipe in the network, 
we split the domain into $N$ finite volume cells $C_j = [x_\jmh,x_\jph]$, where $\dx = x_\jph - x_\jmh$. 
Assume that at time level $t = t^n$, the solution is realized in terms of its cell averages
\begin{equation*}
    \xbar{\bm U}_j^n = \frac{1}{\dx}\int_{C_j} \bm U(x,t^n)\ dx.
\end{equation*}
Then we approximate the contribution of the non-stiff flux terms $\bm R := -\widetilde{\bm F}(\bm U)_x$ in each cell using
\begin{equation}
    \bg{\mc{R}}^n_j = -\frac{\widetilde{\bg{\mc{F}}}^n_\jph - \widetilde{\bg{\mc{F}}}^n_\jmh}{\dx},
    \label{eq:nonstiff}
\end{equation}
where $\widetilde{\bg{\mc{F}}}^n_{j \pm \frac{1}{2}}$ are computed using the CU fluxes; see, e.g., \cite{Harten1993Upstream,Kurganov2001Semidiscrete,Kurganov2002Solution}:
\begin{equation}
    \widetilde{\bg{\mc{F}}}^n_\jph = \frac{
    s^+_\jph \widetilde{\bm F}\p{\bm U^-_\jph} 
    - s^-_\jph \widetilde{\bm F}\p{\bm U^+_\jph}}{s^+_\jph-s^-_\jph}
    +\frac{s^+_\jph s^-_\jph}{s^+_\jph-s^-_\jph}\p{\bm U^+_\jph - \bm U^-_\jph}.
    \label{eq:CUflux}
\end{equation}
Here, $\bm U^\pm_\jph$ are one-sided point values of $\bm U$ at the cell interface, computed at time $t=t^n$ using the piecewise linear reconstruction
\begin{equation*}
    \widetilde{\bm{U}}^n(x) := \xbar{\bm U}^n_j + \p{\bm U_x}^n_j\p{x - x_j}, \quad x \in C_j.
\end{equation*}
Therefore, the values of the reconstruction at the interface $x_\jph$ are
\begin{equation}
    \bm U^-_\jph = \xbar{\bm U}^n_j + \frac{\dx}{2}\p{\bm U_x}^n_j, 
    \qquad 
    \bm U^+_\jph = \xbar{\bm U}^n_{j+1} - \frac{\dx}{2}\p{\bm U_x}^n_{j+1}, 
    \label{eq:UpmHalf}
\end{equation}
where the slopes $\p{\bm U_x}^n_j$ are computed using a nonlinear limiter to avoid oscillations. 
In this paper, we use the generalized minmod limiter (see, e.g., \cite{Lie2003Artificial,Nessyahu1990NonOsc,Sweby1984High}):
\begin{equation*}%\label{eq:limiter}
    \p{\bm U_x}^n_j = {\textrm{minmod}}\p{
    \theta \frac{\xbar{\bm U}^{n}_{j+1}-\xbar{\bm U}^{n}_{j}}{\dx},\ 
    \frac{\xbar{\bm U}^{n}_{j+1}-\xbar{\bm U}^{n}_{j-1}}{2\dx},\ 
    \theta \frac{\xbar{\bm U}^{n}_{j}-\xbar{\bm U}^{n}_{j-1}}{\dx} },
\end{equation*}
where the minmod function is 
\begin{equation*}
\mbox{minmod}(z_1,z_2,\ldots)=\left\{\begin{aligned}
&\min(z_1,z_2,\cdots)&&\mbox{if}~z_i>0,~\forall i,\\
&\max(z_1,z_2,\cdots)&&\mbox{if}~z_i<0,~\forall i,\\
&0&&\mbox{otherwise},
\end{aligned}\right.
%\label{eq:minmod}
\end{equation*}
and is applied component-wise. 
Lastly, $s^\pm_\jph$ of \eqref{eq:CUflux} denote the one-sided speeds of propagation, estimated at time $t=t^n$ using the largest and smallest eigenvalues of the Jacobian $\partial \widetilde{\bm F}/ \partial \bm U$ seen in \eqref{eq:eigen}:
\begin{equation}
    \begin{aligned}
        s^+_\jph =& \max\left\{u^-_\jph + \sqrt{\p{1-\alpha}\p{u^-_\jph}^2 + \frac{\alpha}{\eps^2}\br{ p'\p{\rho^-_\jph}-a^n}},\right.\\
        &\hspace{2cm}\left.u^+_\jph + \sqrt{\p{1-\alpha}\p{u^+_\jph}^2 + \frac{\alpha}{\eps^2}\br{ p'\p{\rho^+_\jph}-a^n}}, 0\right\},\\[6pt]
        s^-_\jph =& \min\left\{u^-_\jph - \sqrt{\p{1-\alpha}\p{u^-_\jph}^2 + \frac{\alpha}{\eps^2}\br{ p'\p{\rho^-_\jph}-a^n}},\right.\\
        &\hspace{2cm}\left.u^+_\jph - \sqrt{\p{1-\alpha}\p{u^+_\jph}^2 + \frac{\alpha}{\eps^2}\br{ p'\p{\rho^+_\jph}-a^n}}, 0\right\}.
    \end{aligned}
    \label{eq:wavespeeds}
\end{equation}

%========================================
\subsection{Fully Discrete Method}\label{sec3.3}
To obtain the fully discrete AP scheme for cell $C_j$, we take the elliptic equation for $\rho^{n+1}$ in \eqref{eq:elliptic_rho} and the equation for $(\rho u)^{n+1}$ in \eqref{eq:semi_ru} and
(i) compute non-stiff flux terms $\bm R^n = -\widetilde{\bm F}(\bm U)_x$ using the CU numerical fluxes described in \S\ref{sec3.2}; 
(ii) discretize the first
spatial derivative terms using standard second-order central difference; and 
(iii) use a standard finite difference approximation for second derivatives on the term $\p{\rho_x^{n+1}/\Psi^n}_x$.
Thus, the fully discrete AP scheme for $\xbar \rho_j^{n+1}$ reads
\begin{equation}
    \begin{aligned}
        &\xbar{\rho}^{n+1}_j - \frac{\dt^2}{\dx^2\eps^2}a^n(1-\alpha)\br{\phi^n_\jph \xbar{\rho}^{n+1}_{j+1} - \p{\phi^n_\jph + \phi^n_\jmh}\xbar{\rho}^{n+1}_j + \phi^n_\jmh \xbar{\rho}^{n+1}_{j-1}}\\
        &\hspace{4cm} = \xbar{\rho}^n_j+ \dt \mc{R}^{\rho,n}_j - \dt (1-\alpha) D_x \br{\frac{1}{\Psi^n_j}\p{\p{\xbar{\rho u}}^n_j + \dt \mc{R}^{\rho u, n}_j} },
    \end{aligned}
    \label{eq:discrete_rho}
\end{equation}
where $\bg{\mc{R}}^n_j = (\mc R^{\rho,n}_j, \mc R^{\rho u, n}_j)^\top$ is defined in \eqref{eq:nonstiff}--\eqref{eq:CUflux}, $\Psi_j^n = \Psi^n(x_j)$ from \eqref{eq:psi},
$$D_x\xi_j = \frac{\xi_{j+1}-\xi_{j-1}}{2\dx}$$
is the discrete central difference operator, and 
$$\phi^n_\jph = \frac{1}{2}\br{\frac{1}{\Psi^n_j} + \frac{1}{\Psi^n_{j+1}}}$$
is used for the evaluation of $1/\Psi_\jph^n$ to keep the second-order spatial approximation in the elliptic solve. 
Note that the definition of $\Psi_j^n$ in \eqref{eq:psi} results in $0 < \phi^n_\jph \leq 1$, in turn implying that the corresponding matrix to be formed on the right hand side of \eqref{eq:discrete_rho} is non-singular (by Gershgorin Theorem).

After solving the linear system in \eqref{eq:discrete_rho} for $\set{\xbar \rho_j^{n+1}}$, it can be directly substituted into the spatially discretized version of \eqref{eq:semi_ru} to obtain the fully discrete approximation for $\p{\xbar{\rho u}}^{n+1}_j$:
\begin{equation}
    \p{\xbar{\rho u}}^{n+1}_j = \frac{1}{\Psi^n_j}\br{\p{\xbar{\rho u}}^n_j + \dt \mc{R}^{\rho u, n}_j - \frac{a^n \dt}{\eps^2}D_x\xbar\rho^{n+1}_j}.
    \label{eq:discrete_ru}
\end{equation}

%===================================
\subsection{Numerical Diffusion and Stability}\label{sec3.4}
In this section, we describe how the proposed method addresses the numerical difficulties of explicit schemes previously discussed in \S\ref{sec2.1}.

The stability of the proposed AP scheme is ensured by Lemma 3.1 of \cite{Haack2012AllSpeed}, which states that if each piece (the implicit and explicit portions) of the split method is stable, then the full method is also stable. 
For the fast (stiff) dynamics, one can show that implicit discretization in time, 
%via the influence of Resenbrock-type Runge-Kutta methods
seen in \eqref{eq:IMEXrho}--\eqref{eq:IMEXru}, is stable for all $\rho^n$ and $u^n$.
Thus, stability of the proposed AP scheme is controlled by the time discretization of the slow (non-stiff) dynamics. 
The CFL condition for the slow dynamics can be found using the eigenvalues from \eqref{eq:eigen}, resulting in the following time-step restriction:
\begin{equation}
    \dt_{\rm{AP}} \leq \nu \frac{\dx}{\displaystyle{\max_x}\set{\abs{u} + \sqrt{\p{1-\alpha}u^2 + \frac{\alpha}{\eps^2}\br{p'(\rho)-a(t)}}}},
    \label{eq:CFL_ap}
\end{equation}
in which, due to the selection of $\alpha$ and $a(t)$ made in \eqref{eq:alpha_a}, the denominator is independent of $\eps$. 

To ensure $\dt_{\rm{AP}}$ is completely $\eps$-independent, we must also confirm the numerical diffusion of the proposed scheme does not have a $\eps^{-1}$ dependence or worse. 
To do this, we analyze the leading order of numerical diffusion of the CU spatial discretization described in \S\ref{sec3.2}, and substitute this into the full discretization in equations \eqref{eq:discrete_rho}--\eqref{eq:discrete_ru}.
We start by rewriting the CU flux in \eqref{eq:CUflux} in the form 
\begin{equation*}%\label{eq:CUflux_dif}
    \widetilde{\bg{\mc{F}}}^n_\jph = \frac{\widetilde{\bg{F}}(\xbar{\bm U}_j^n)+\widetilde{\bg{F}}(\xbar{\bm U}_{j+1}^n)}{2} + \bg{\mc{D}}^n_\jph,
\end{equation*}
where 
\begin{equation}\label{eq:diff}
    \begin{aligned}
        \bg{\mc{D}}^n_\jph =& \frac{s^+_\jph}{2\p{s^+_\jph-s^-_\jph}}\br{2 \widetilde{\bg{F}}(\bm U^-_\jph)-\widetilde{\bg{F}}(\xbar{\bm U}_j^n)-\widetilde{\bg{F}}(\xbar{\bm U}_{j+1}^n)}\\[4pt]
        &-\frac{s^-_\jph}{2\p{s^+_\jph-s^-_\jph}}\br{2 \widetilde{\bg{F}}(\bm U^+_\jph)-\widetilde{\bg{F}}(\xbar{\bm U}_j^n)-\widetilde{\bg{F}}(\xbar{\bm U}_{j+1}^n)}+\frac{s^+_\jph s^-_\jph}{s^+_\jph-s^-_\jph}\p{\bm U^+_\jph - \bm U^-_\jph},
    \end{aligned}
\end{equation}
is the numerical diffusion term of the CU discretization, with $\bg{\mc{D}}_\jph^n := (\mc D^{\rho,n}_\jph, \mc D^{\rho u,n}_\jph)^\top$ to denote the different components. 
Here, $s^\pm_\jph$ are defined in \eqref{eq:wavespeeds} and $\bm U^\pm_\jph$ are defined in \eqref{eq:UpmHalf}. 
%We will also use the notation $\bg{\mc{D}} := (\mc D^\rho, \mc D^{\rho u})^\top$ to denote the different components of diffusion.
To compute the leading order of the numerical diffusion in \eqref{eq:diff}, we introduce the following asymptotic expansions:
\begin{align*}
    \xbar \rho^n_j & = \rho^{(0),n}_j + \eps^2 \rho^{(1),n}_j + \cdots,\\
    \rho^\pm_\jph & = \rho^{(0),\pm}_\jph + \eps^2 \rho^{(1),\pm}_\jph + \cdots.
\end{align*}
The asymptotic expansion of $u_j^n$ and $u^\pm_\jph$ follows analogously, and the discrete asymptotic expansion of the pressure term follows the same form as that in equation \eqref{eq:A_expan_pr}.
In addition, we expand $a^n$ in a way that follows its definition in \eqref{eq:alpha_a}:
$$
a^n = \min_x \gamma \p{\rho_j^{(0),n}}^{\gamma -1} + \eps^2 a^{(2),n} + \cdots,
$$
where $a^{(2),n}$ is a constant in space at time $t^n$.
Using these expansions, we can obtain the leading order of the numerical diffusion \eqref{eq:diff} for $\rho$:
\begin{equation}\label{eq:diff_rho}
    \begin{aligned}
        \mc{D}^{\rho,n}_\jph =& \frac{\alpha s^+_\jph}{2\p{s^+_\jph-s^-_\jph}}\br{2 \rho^{(0),-}_\jph u^{(0),-}_\jph-\rho^{(0),n}_j u^{(0),n}_j-\rho^{(0),n}_{j+1} u^{(0),n}_{j+1}}\\[4pt]
        &-\frac{\alpha s^-_\jph}{2\p{s^+_\jph-s^-_\jph}}\br{2 \rho^{(0),+}_\jph u^{(0),+}_\jph-\rho^{(0),n}_j u^{(0),n}_j-\rho^{(0),n}_{j+1} u^{(0),n}_{j+1}}\\
        &+\frac{s^+_\jph s^-_\jph}{s^+_\jph-s^-_\jph}\p{\rho^{(0),+}_\jph - \rho^{(0),-}_\jph}.
    \end{aligned}
\end{equation}
Since $\alpha = \eps^b,\ b \geq 2$ from \eqref{eq:alpha_a}, this is $\mc O(1)$ by the last term of \eqref{eq:diff_rho}, as $\rho$ is generally not constant in space.
Similarly, one can compute the leading order of numerical diffusion \eqref{eq:diff} for $\rho u$, which comes from the $\mc O(\eps^{-2})$ portion of the flux $\widetilde{\bm F}(\bm U)$ seen in equation \eqref{eq:newvecform}. 
However, since $a^n$ and its expansion are related to $p'(\rho)$ and $p(\rho) = \rho^\gamma$, it is clear that the expansions to compute the diffusion will not result in any cancellation of the $\eps^{-2}$ term. 
Thus, $\mc{D}^{\rho u,n}_\jph = \mc O(\eps^{-2})$.
%Normally this would bring about concern as we want the leading order of diffusion to be $\eps$-independent, but as we will shortly detail for the following reason. 

% \begin{equation}\label{eq:diff_ru}
%     \begin{aligned}
%         \mc{D}^{\rho u,n}_\jph =& \frac{1}{\eps^2}\cdot \frac{s^+_\jph}{2\p{s^+_\jph-s^-_\jph}}\br{2 \p{\rho^{(0),-}_\jph}^\gamma - \min_x \gamma \p{\rho_j^{(0),n}}^{\gamma -1} - \p{\rho^{(0),n}_j}^\gamma -
%         \p{\rho^{(0),n}_{j+1}}^\gamma}\\[4pt]
%         &-\frac{1}{\eps^2}\cdot\frac{ s^-_\jph}{2\p{s^+_\jph-s^-_\jph}}\br{2 \rho^{(0),+}_\jph u^{(0),+}_\jph-\rho^{(0),n}_j u^{(0),n}_j-\rho^{(0),n}_{j+1} u^{(0),n}_{j+1}}\\
%         &+\frac{s^+_\jph s^-_\jph}{s^+_\jph-s^-_\jph}\p{\rho^{(0),+}_\jph u^{(0),+}_\jph - \rho^{(0),-}_\jph u^{(0),-}_\jph}.
%     \end{aligned}
% \end{equation}

Finally, substituting the CU fluxes into \eqref{eq:nonstiff}, we can obtain the following approximations to $\bg{\mc R}_j^n$: 
\begin{align}
    \mc R_j^{\rho,n} &= - D_x\widetilde{F}^\rho(\bm U_j^n) + \mc O(\dx^2), \label{eq:R_rho}\\[4pt]
    \mc R_j^{\rho u,n} &= - D_x\widetilde{F}^{\rho u}(\bm U_j^n) + \mc O\p{\frac{\dx^2}{\eps^2}}, \label{eq:R_ru}
\end{align}
where $\widetilde{\bm F} := (\widetilde{F}^\rho, \widetilde{F}^{\rho u})^\top$ and $D_x$ again represents the central difference operator. 
Here, the $\mc O(\dx^2)$ term in \eqref{eq:R_rho} and the $\mc O(\eps^{-2}\dx^2)$ term in \eqref{eq:R_ru} come from the leading orders from the diffusion $\mc{D}^{\rho,n}_\jph$ and $\mc{D}^{\rho u,n}_\jph$, respectively, and the $\dx^2$ piece arises since these diffusion terms are introduced via the second-order CU discretization. 

Normally, this $\mc O(\eps^{-2}\dx^2)$ diffusion term is concerning.
However, if we substitute \eqref{eq:R_ru} into the fully discrete AP scheme seen \eqref{eq:discrete_rho} and \eqref{eq:discrete_ru}, we see that this $\mc O(\eps^{-2}\dx^2)$ diffusion term is always paired with a division by $\Psi_j^n$. 
Then, since $\Psi_j^n$ defined in \eqref{eq:psi} is $\mc O(\eps^{-2})$, the dependence of $\eps$ in the leading order diffusion term cancels.
%then the largest contribution of the numerical diffusion are  $\mc O(\dx^2)$ (from the $\mc R_j^{\rho,n}$ contribution) and $\mc O(\eps^{-2}\dx^2)/\Psi_j^n$ (from the $\mc R_j^{\rho u,n}$ contribution).
%However, since $\Psi_j^n$ defined in \eqref{eq:psi} is $\mc O(\eps^{-2})$, the dependence of $\eps$ in the leading order diffusion term cancels.
This $\eps$-independence within the numerical diffusion, in combination with the denominator of \eqref{eq:CFL_ap} being independent of $\eps$, implies that the time-step restriction of the proposed AP scheme is indeed independent of $\eps$. 
Furthermore, it is sufficient to take $\dt_{\rm AP} = \mc O(\dx)$ (as opposed to the $\mc O(\eps\dx)$ restriction on explicit schemes) to enforce stability.

%===================================
\subsection{Boundary Conditions on Pipe Networks}\label{sec3.5}
\subsubsection{Coupling Conditions at Pipe Intersections}\label{sec3.5.1}
The algorithm described in \S\ref{sec3.1}--\S\ref{sec3.3} is applied to every pipe within the network of interest. 
To extend to a full network, we must implement coupling conditions \eqref{eq:cons_ru}, with one choice of \eqref{eq:eq_p}, \eqref{eq:eq_ru} or \eqref{eq:p_loss}, on the boundaries of each pipe that meets at the intersections. 
To obtain the boundary values corresponding to the junction, we solve the so-called half-Riemann problem: 
Assume a set of coupling conditions from \S\ref{sec2.2new} and given constant initial data $(\bm U^0)^{(k)}$ for each pipe $k$. 
We then seek the values $\bm U^*$ satisfying the coupling conditions such that the half-Riemann problem~\cite{herty2006coupl,herty2006coupl,garavello2016model} 
\begin{equation}\label{eq:half_R}
    \begin{aligned}
        & \bm U^{(k)}_t + \bm F(\bm U^{(k)})_x = \bm 0, \\[2pt]
        &\bm U(x,0) = \left\{
        \begin{aligned}
            &\bm U^*, && \textrm{if } x \leq 0,\\
            & (\bm U^0)^{(k)},&& \textrm{if } x > 0,
        \end{aligned}\right.
    \end{aligned}
\end{equation}
admits a self-similar solution in which all generated waves have \textit{non-negative} speeds. 
Here, $\bm U$ and $\bm F(\bm U)$ are defined in \eqref{eq:vecform}.
Then, if considering coupling conditions \eqref{eq:cons_ru} and \eqref{eq:eq_p}, this amounts to solving the following nonlinear system for $(\rho^*)^{(k)}$:
\begin{equation}\label{eq:Lax_sys}
    \begin{aligned}
        &\sum_{k = 1}^K L_2^-\p{(\rho^*)^{(k)}; (\rho^n)^{(k)},L_1^+\p{(\rho^n)^{(k)};(\bm U^0)^{(k)}} } = 0, \\[4pt]
        &p\p{(\rho^*)^{(k)}}%(x_{N+\frac{1}{2}}^{(k)},t)} 
        = p\p{(\rho^*)^{(\ell)}}%(x_{\frac{1}{2}}^{(\ell)},t)}
        , \qquad k \neq \ell,
    \end{aligned} 
\end{equation}
where the known states are $(\bm U^0)^{(k)}$ and $(\bm U^n)^{(k)}$, and the functions $L_{1,2}^\pm$ denote the forward and reversed 1-Lax curve and 2-Lax curve for the isothermal Euler equations in equation \eqref{eq:euler_eps} with $p = \rho^\gamma$ (see similarly, \cite{Colombo2006Well}):
\begin{equation}\label{eq:Lax}
\begin{aligned}
    L_1^+\p{\rho; \widehat{\rho}, \widehat{q}}
    &=
    \left\{
        \begin{aligned}
            & \frac{\rho}{\widehat{\rho}} \ \widehat{q} - \frac{2}{\gamma - 1} \rho \cdot \frac{1}{\eps}\p{\sqrt{p'(\rho)}-\sqrt{p'(\widehat{\rho})}}, && \textrm{if } \rho < \widehat{\rho},\\[4pt]
            &\frac{\rho}{\widehat{\rho}} \ \widehat{q} - \frac{1}{\eps}\sqrt{\frac{\rho}{\widehat{\rho}}\p{\rho - \widehat{\rho}}\br{p(\rho)-p(\widehat{\rho})} }, &&\textrm{if } \rho > \widehat{\rho},
        \end{aligned}\
    \right.\\[8pt]
    L_2^-\p{\rho; \widehat{\rho}, \widehat{q}}
    &=
    \left\{
        \begin{aligned}
            & \frac{\rho}{\widehat{\rho}} \ \widehat{q} - \frac{2}{\gamma - 1} \rho \cdot \frac{1}{\eps}\p{\sqrt{p'(\widehat{\rho})}-\sqrt{p'(\rho)}}, && \textrm{if } \rho < \widehat{\rho}.\\[4pt]
            &\frac{\rho}{\widehat{\rho}} \ \widehat{q} + \frac{1}{\eps}\sqrt{\frac{\rho}{\widehat{\rho}}\p{\rho - \widehat{\rho}}\br{p(\rho)-p(\widehat{\rho})} },
            &&\textrm{if } \rho > \widehat{\rho}, 
        \end{aligned}\
    \right.
\end{aligned}
\end{equation}
The construction of system \eqref{eq:Lax_sys} would follow analogously if instead using the coupling conditions in \eqref{eq:eq_ru} or \eqref{eq:p_loss}.

% \mh{I think, we need to discuss what happens for $\epsilon \to 0$ here. Because the formula's formally are not fulfilled for $\epsilon=0.$
% }

The known states $(\bm U^n)^{(k)}$ seen in \eqref{eq:Lax_sys} are selected by taking the 
cell value nearest the junction at time $t = t^n$; i.e., we take $(\bm U^n)^{(k)} = (\bm U^n_{N})^{(k)}$ for ingoing pipelines and $(\bm U^n)^{(k)} = (\bm U^n_{1})^{(k)}$ for outgoing pipelines. 
The nonlinear system \eqref{eq:Lax_sys} is solved via Newton's method with initial guess $\bm U^*= (\bm U^n)^{(k)}$, in which the Jacobian within the Newton iteration is invertible if the initial guess is subsonic \cite{Colombo2006Well}.
Note the $1/\eps$ dependence in the 1-Lax curve and 2-Lax curve only arises due to considering the non-dimensionalized system \eqref{eq:vecform}, and does not destroy the nonlinear solve of \eqref{eq:Lax_sys}.
This is further verified in Lemma \ref{lma3.2}.
For the numerical experiments conducted in this paper, we observed convergence to a tolerance of $10^{-8}$ within $\sim$2--3 Newton iterations. 

Once $(\rho^*)^{(k)}$ is computed, we can directly calculate $(\rho^* u^*)^{(k)}$ using 
$$
(\rho^* u^*)^{(k)} = L_2^-\p{(\rho^*)^{(k)}; (\rho^n)^{(k)},L_1^+\p{(\rho^n)^{(k)};(\bm U^0)^{(k)}} }. 
$$
The results of the nonlinear solve $(\rho^*)^{(k)}$ and $(\rho^* u^*)^{(k)}$ are then prescribed as the ghost cell values in the spatial discretization described in \S\ref{sec3.2}.

\begin{lemma}\label{lma3.2}
    In the limiting case $\eps \rightarrow 0$, the nonlinear system \eqref{eq:Lax_sys} with Lax curve definitions \eqref{eq:Lax} remains a well-posed problem.  
\end{lemma}
\begin{proof}
    Since the second condition of \eqref{eq:Lax_sys} is $\eps$-independent, we only need to focus on the first condition. 
    For brevity, let us first rewrite the Lax-curve definitions in \eqref{eq:Lax} as
    $$
    L_1^+\p{\rho; \widehat{\rho}, \widehat{q}}
    =
    \left\{
        \begin{aligned}
            & \frac{\rho}{\widehat{\rho}} \ \widehat{q} - \frac{1}{\eps}g(\rho,\widehat{\rho}), && \textrm{if } \rho < \widehat{\rho},\\[4pt]
            &\frac{\rho}{\widehat{\rho}} \ \widehat{q} - \frac{1}{\eps}h(\rho,\widehat{\rho}), &&\textrm{if } \rho > \widehat{\rho},
        \end{aligned}\
    \right.
    \quad\textrm{and}\quad 
    L_2^-\p{\rho; \widehat{\rho}, \widehat{q}}
    =
    \left\{
        \begin{aligned}
            & \frac{\rho}{\widehat{\rho}} \ \widehat{q} + \frac{1}{\eps}g(\rho,\widehat{\rho}), && \textrm{if } \rho < \widehat{\rho},\\[4pt]
            &\frac{\rho}{\widehat{\rho}} \ \widehat{q} + \frac{1}{\eps}h(\rho,\widehat{\rho}), &&\textrm{if } \rho > \widehat{\rho},
        \end{aligned}\
    \right.
    $$
    where $g(\rho,\widehat{\rho})$ and $h(\rho,\widehat{\rho})$ are taken such that these are equivalent to \eqref{eq:Lax}.
    In this proof, we will only consider the case in which $\rho > \widehat{\rho}$ for both $L_1^+$ and $L_2^-$, and the proofs for the other three pairings follow analogously. 
    Following this assumption, then the Lax curves for this case read
    $$
    L_1^+\p{\rho; \widehat{\rho}, \widehat{q}} = \frac{\rho}{\widehat{\rho}} \ \widehat{q} - \frac{1}{\eps}h(\rho,\widehat{\rho})
    \qquad \textrm{and} \qquad
    L_2^-\p{\rho; \widehat{\rho}, \widehat{q}} = \frac{\rho}{\widehat{\rho}} \ \widehat{q} + \frac{1}{\eps}h(\rho,\widehat{\rho}). 
    $$
    Note that in the context of the nonlinear system \eqref{eq:Lax_sys}, the case in which $\rho > \widehat{\rho}$ for both $L_1^+$ and $L_2^-$ is equivalent to having $(\rho^0)^{(k)} < (\rho^n)^{(k)}$ and $(\rho^n)^{(k)} < (\rho^*)^{(k)}$. 
    We can then use this specific case of the Lax curves and substitute directly into \eqref{eq:Lax_sys} to obtain the nonlinear system: 
    \begin{equation}\label{eq:syseps}
        \begin{aligned}
            &\sum_{k = 1}^K \frac{(\rho^*)^{(k)}}{(\rho^n)^{(k)}}
            \br{\frac{(\rho^n)^{(k)}}{(\rho^0)^{(k)}} (q^0)^{(k)} -\frac{1}{\eps}h\p{(\rho^n)^{(k)},(\rho^0)^{(k)}} }
            +\frac{1}{\eps}h\p{(\rho^*)^{(k)},(\rho^n)^{(k)}} = 0, \\[4pt]
            &p\p{(\rho^*)^{(k)}}%(x_{N+\frac{1}{2}}^{(k)},t)} 
            = p\p{(\rho^*)^{(\ell)}}%(x_{\frac{1}{2}}^{(\ell)},t)}
            , \qquad k \neq \ell.
        \end{aligned}
    \end{equation}
    % $$
    % \sum_{k = 1}^K \frac{(\rho^*)^{(k)}}{(\rho^n)^{(k)}}
    % \br{\frac{(\rho^n)^{(k)}}{(\rho^0)^{(k)}} (q^0)^{(k)} -\frac{1}{\eps}h\p{(\rho^n)^{(k)},(\rho^0)^{(k)}} }
    % +\frac{1}{\eps}h\p{(\rho^*)^{(k)},(\rho^n)^{(k)}} = 0.
    % $$ 
    This system with $\eps > 0$ has already been proven to be a well-posed system in \cite{Colombo2006Well}.
    We then multiply the first equation of \eqref{eq:syseps} by $\eps$ and take the limit $\eps \rightarrow 0$ to obtain an equivalent nonlinear system in which we wish to solve for $(\rho^*)^{(k)}$, which reads 
    \begin{equation}\label{eq:syslim}
        \begin{aligned}
            &\sum_{k = 1}^K 
            h\p{(\rho^*)^{(k)},(\rho^n)^{(k)}}
            -\frac{(\rho^*)^{(k)}}{(\rho^n)^{(k)}}
            h\p{(\rho^n)^{(k)},(\rho^0)^{(k)}} = 0, \\[4pt]
            &p\p{(\rho^*)^{(k)}}%(x_{N+\frac{1}{2}}^{(k)},t)} 
            = p\p{(\rho^*)^{(\ell)}}%(x_{\frac{1}{2}}^{(\ell)},t)}
            , \qquad k \neq \ell.
        \end{aligned} 
    \end{equation}
    This is equivalent to that of \eqref{eq:syseps} with $(q^0)^{(k)}$ being taken as zero. 
    Therefore, this is also already proven as a well-posed problem by the work in \cite{Colombo2006Well}.
    Hence, the system \eqref{eq:Lax_sys} is well-posed in the limiting case $\eps \rightarrow 0$.
\end{proof}

\begin{rmk}\label{rmk3.3}
    As stated in \S\ref{sec2.2}, the other coupling conditions  \eqref{eq:eq_ru} and \eqref{eq:p_loss} reduce to the pressure balance coupling condition in \eqref{eq:eq_p} in the limiting case $\eps \rightarrow 0$.
    Thus, Lemma \ref{lma3.2} is additionally valid for coupling conditions \eqref{eq:eq_ru} and \eqref{eq:p_loss}, as their corresponding nonlinear systems would also reduce to \eqref{eq:syslim} in the low Mach/high friction limit.
\end{rmk}

\begin{rmk}
There exists also other approaches to discretize the coupling conditions. Here, we mention the possibility to do so fully implicit in time~\cite{kolb2010}, using a finite--element  based approaches~\cite{egger2018}, linearization approaches ~\cite{banda2016numer,borsche2014aderschemhigh}, or central scheme type discretizations \cite{herty2023centr}. As the previous Lemma shows the treatment using half-Riemann problems is, however, sufficient to guarantee here the asymptotic preserving property of the scheme. We therefore do not investigate those alternative approaches. 
\end{rmk}

\subsubsection{Additional Boundary Condition Requirements}\label{sec3.5.2}
In addition to defining the ghost cell values for $\rho$ and $\rho u$ at the single point where the pipes meet, the algorithm described in \S\ref{sec3.1}--\S\ref{sec3.3} requires defined values for (i) the reconstructed values on both sides of the cell interfaces $x_{\frac{1}{2}}^{(k)}$ and $x_{N+\frac{1}{2}}^{(k)}$; and (ii) the central difference derivative of numerical fluxes near pipeline boundaries. 

The requirement for reconstructed values at the domain boundaries $x_{\frac{1}{2}}^{(k)}$ and $x_{N+\frac{1}{2}}^{(k)}$ of each pipeline is standard and are directly needed within the numerical flux evaluations; see \eqref{eq:CUflux}.
These evaluations are trivial at the inlets and outlets of the pipe network, as they typically involve Dirichlet or Neumann boundary conditions. 
Again the difficulty lies at the junctions of the network. 
Since we do not have enough information to form a piecewise-linear reconstruction (via equation \eqref{eq:UpmHalf}) of the ghost cells sitting at the junctions, we instead opt for the first-order reconstruction; that is, at pipe intersections, we take
\begin{equation*}%\label{eq:rec_jun}
    (\bm U^\pm_{N+\frac{1}{2}})^{(k)} = \bm U^*
    \qquad \textrm{and} \qquad 
    (\bm U^\pm_{\frac{1}{2}})^{(k)} = \bm U^*,
\end{equation*}
for ingoing and outgoing pipelines, respectively. 
Here, $\bm U^*$ denotes the solution to nonlinear system \eqref{eq:Lax_sys}.

The other boundary condition issue lies in the central difference within the fully discrete method to calculate $\xbar \rho^{n+1}_j$. 
As seen in \eqref{eq:discrete_rho}, we require a central difference of $\mc R_j^{\rho u,n}$, which by \eqref{eq:nonstiff} implies the need to evaluate the unknowns $\widetilde{\bg{\mc{F}}}^n_{-\frac{1}{2}}$ and $\widetilde{\bg{\mc{F}}}^n_{N+\frac{3}{2}}$ for each pipe.
This would require even more ghost cell evaluations than the reconstructions on the boundaries.
Thus, we instead assume a zero-extrapolation of $\mc R_j^{\rho u,n}$ for cells $C_0$ and $C_{N+1}$ and use the extrapolated values in the central difference seen in \eqref{eq:discrete_rho}; that is, we use
%Thus, instead of using a central difference of $\mc R_j^{\rho u,n}$ in cells $C_1$ and $C_n$, we opt for a forward and backward difference in these cells, respectively; that is, when discretizing the right hand side of \eqref{eq:discrete_rho} we instead use the following for each pipe:
$$
D_x\br{\frac{\mc R_1^{\rho u,n}}{\Psi_1^n}} = \frac{1}{\dx}\br{\frac{\mc R_2^{\rho u,n}}{\Psi_2^n} - \frac{\mc R_1^{\rho u,n}}{\Psi_0^n}}
\qquad \textrm{and} \qquad 
D_x\br{\frac{\mc R_N^{\rho u,n}}{\Psi_N^n}} = \frac{1}{\dx}\br{\frac{\mc R_N^{\rho u,n}}{\Psi_{N+1}^n} - \frac{\mc R_{N-1}^{\rho u,n}}{\Psi_{N-1}^n}}.
$$
%{\color{red}[Check that this is actually done in the code, specifically, the $\Psi$ terms]}
Both this and the first-order reconstructions at pipe intersections imply a first-order method at the junctions, and thus everywhere in the domain. 
In the near future, we intend to expand upon this, the first-order time discretization, and the piecewise-constant initial data within the half-Riemann problem \eqref{eq:half_R} to enhance the proposed scheme to be fully second-order accurate.

% {\color{cyan}
% \begin{itemize}
%     %\item assumptions
%     %\item state half riemann ICs
%     %\item nonlinear system(s) we need to solve
%     %\item lax curve equations, note eps dependence
%     \item discuss how this does not induce issues with AP?
%     %\item Discuss flux computation in the ghost cell
% \end{itemize}
% }

%===================================
\subsection{The Discrete Low Mach/High Friction Limit}\label{sec3.6}
%{\color{red}[Is this even needed?]}
In this section, we show that in the limit as $\eps \rightarrow 0$, the method converges to the asymptotic state in \eqref{eq:A_Behav}. 
Consider the following asymptotic expansions for the discretization of $\rho$, $u$ and $p$, which follow from \S\ref{sec2.2}:
\begin{equation}
    \begin{aligned}
        \xbar\rho_j^n &= \xbar\rho^{(0),n}_j + \eps^2 \xbar\rho^{(2),n}_j + \cdots, \\
        u_j^n &= u^{(0),n}_j + \eps^2 u^{(2),n}_j + \cdots,\\
        p(\xbar\rho_j^n) &= \p{\rho^{(0),n}_j}^\gamma + \eps^2 \gamma \p{\rho^{(0),n}_j }^{\gamma-1}\rho^{(2),n}_j+ \cdots,
    \end{aligned}
    \label{eq:A_discrete}
\end{equation}
and analogously at time $t^{n+1}$. 
We want to show that as the proposed method advances in time, it keeps the asymptotic state up to the predicted order of accuracy.
To this end, we wish to find a bound on the asymptotic expansion for the time-advanced solution
$$
\p{\xbar\rho_j^{(0),n+1}}^\gamma+\frac{C_\delta \kappa}{2}\xbar{\rho}^{(0)n+1}_j u^{(0),n+1}_j.
$$
Since the pressure to friction balance in the asymptotic state arises from the momentum equation in \eqref{eq:euler_eps}, we start with a slightly manipulated form of the discrete approximation for $\p{\xbar{\rho u}}^{n+1}_j$ in \eqref{eq:discrete_ru}:
$$
\frac{\eps^2}{\dt}\p{\p{\xbar{\rho u}}^n_j + \dt \mc{R}^{\rho u, n}_j - \frac{a^n \dt}{\eps^2}D_x\xbar\rho^{n+1}_j-\Psi^n_j\p{\xbar{\rho u}}^{n+1}_j } = 0.
$$
Note that the left hand side is effectively the local truncation error multiplied by the $\mc O(1)$ term $\eps^2\Psi^n_j$.
Here, we first substitute in for $\mc{R}^{\rho u, n}_j$ using equation \eqref{eq:R_ru}, resulting in the form
$$
\frac{\eps^2}{\dt}\p{\p{\xbar{\rho u}}^n_j - \dt D_x\widetilde{F}^{\rho u}(\bm U_j^n) - \frac{a^n \dt}{\eps^2}D_x\xbar\rho^{n+1}_j-\Psi^n_j\p{\xbar{\rho u}}^{n+1}_j }
= \frac{\eps^2}{\dt}\cdot C_1\frac{\dt \dx^2}{\eps^2},
$$
where $C_1$ comes from the leading order of diffusion $\mc D_\jph^{\rho u, n}$ that appears within the $\mc O(\eps^{-2}\dx^2)$ term of \eqref{eq:R_ru}.
Substituting in $\widetilde{F}^{\rho u}(\bm U_j^n)$ from \eqref{eq:splitflux} and $\Psi_j^n$ from \eqref{eq:psi}, we obtain 
$$
\eps^2\br{\frac{\p{\xbar{\rho u}}^n_j-\p{\xbar{\rho u}}^{n+1}_j}{\dt}-D_x\p{\rho u^2}_j^n} - D_x\p{p(\xbar\rho_j^n)-a^n \xbar\rho_j^n} - a^n D_x\xbar\rho^{n+1}_j-\frac{C_\delta \kappa}{2}\p{\xbar{\rho u}}^{n+1}_j
= C_1\dx^2.
$$
We then apply the asymptotic expansions from \eqref{eq:A_discrete} and take $\eps \rightarrow 0$, resulting in 
$$
D_x \br{\p{\xbar\rho_j^{(0),n}}^\gamma}+\frac{C_\delta \kappa}{2}\xbar{\rho}^{(0)n+1}_j u^{(0),n+1}_j= a^n D_x\p{\xbar\rho^{(0),n}_j-\xbar\rho^{(0),n+1}_j} - C_1\dx^2. 
$$
Since $D_x$ represents the second-order central difference operator, then we equivalently have
$$
\br{\p{\xbar\rho_j^{(0),n}}^\gamma}_x+\frac{C_\delta \kappa}{2}\xbar{\rho}^{(0)n+1}_j u^{(0),n+1}_j= a^n \p{\xbar\rho^{(0),n}_j-\xbar\rho^{(0),n+1}_j}_x +(C_2 - C_1)\dx^2,
$$
where $C_2$ represents the the leading error coefficient from using central difference.
Lastly, applying Taylor series to the terms at time $t^n$ about time $t^{n+1}$, we obtain
$$
\br{\p{\xbar\rho_j^{(0),n+1}}^\gamma}_x+\frac{C_\delta \kappa}{2}\xbar{\rho}^{(0)n+1}_j u^{(0),n+1}_j= C_3\dt+(C_2 - C_1)\dx^2.
$$
Here, $C_3$ denotes the leading error term from the Taylor series. 
Therefore, we can bound the asymptotic expansion at time $t^{n+1}$:
\begin{equation*}%\label{eq:bound} 
\norm{\br{\p{\xbar\rho_j^{(0),n+1}}^\gamma}_x+\frac{C_\delta \kappa}{2}\xbar{\rho}^{(0)n+1}_j u^{(0),n+1}_j} \leq C(\dt+\dx^2),
\end{equation*}
with $C = \max\set{C_3, C_2-C_1}$, 
hence implying that the proposed method preserves the asymptotic state up to the order of the scheme. 

The limiting scheme is accompanied by the (limiting) coupling conditions \eqref{eq:Lax_sys}  at the boundary implemented again using  ghost cells. 
Note again that by the discussion made in \S\ref{sec3.5.2}, the pipe entries and exits occurring at intersections would have a bound that is instead first-order in both space and time. 

%simplemented on the boundaries will at most have the same order of error as that of the neighboring cells, as long as the Newton solver tolerance is smaller than the numerical errors of the method.
%Furthermore, the values calculated at the junctions also converge with respect to $\dt$ and $\dx$ in the limiting case $\eps \rightarrow 0$.
%{\color{red} [Is this paragraph needed?] }

%\newpage
%===================================
\section{Numerical Results}\label{sec4}
In this section, we demonstrate the accuracy and performance of the proposed AP scheme on pipe networks in several numerical experiments conducted on the isentropic Euler equations in \eqref{eq:euler_eps}. 
For all examples, we use an adiabatic index value of $\gamma = 5/3$, follow the CFL condition in \eqref{eq:CFL_ap} and take the CFL number $\nu = 0.45$, in the source term of \eqref{eq:euler_eps} choose $C_\delta = 1$ and $\kappa = 10^{-3}$, take the minmod parameter $\theta = 1.3$, and set the tolerance for the Newton solve at the junctions to be $10^{-8}$.

In this paper, we will only conduct verifications on the two types of T-junctions -- the so-called 1-to-2 T-junction, which has one pipe entering and two pipes leaving the junction, and the so-called 2-to-1 T-junction, which has two pipes entering and one pipe leaving the junction. 
Illustrations of these are presented in Figure \ref{fig:t_junctions}.
While we only consider the example T-junction networks with all pipes having the same length, the method is of course not restricted to this and can be extended to large and more complex pipe networks.
\begin{figure}[ht!]
\centering
\resizebox{.3\textwidth}{!}{%
\begin{circuitikz}
\tikzstyle{every node}=[font=\LARGE]
%\draw (5.75,13.5) to[short] (5.75,13.5);
\node at (9.25,14) [circ] {};
\draw (6.5,14) to[short] (12,14);
\draw (9.25,14) to[short] (9.25,11.25);
\draw [->, >=Stealth] (7.84,14) -- (7.85,14);
\draw [->, >=Stealth] (9.25,12.51) -- (9.25,12.5);
\draw [->, >=Stealth] (10.7,14) -- (10.71,14);
\end{circuitikz}
}%
\hspace{1cm}
\resizebox{.3\textwidth}{!}{%
\begin{circuitikz}
\tikzstyle{every node}=[font=\LARGE]
%\draw (5.75,13.5) to[short] (5.75,13.5);
\node at (9.25,14) [circ] {};
\draw (6.5,14) to[short] (12,14);
\draw (9.25,14) to[short] (9.25,11.25);
\draw [->, >=Stealth] (7.84,14) -- (7.85,14);
\draw [->, >=Stealth] (9.25,12.51) -- (9.25,12.5);
\draw [->, >=Stealth] (10.6,14) -- (10.59,14);
\end{circuitikz}
}%
\caption{\sf Illustrations of the 1-to-2 T-junction (left), and 2-to-1 T-junction (right). }
\label{fig:t_junctions}
\end{figure}
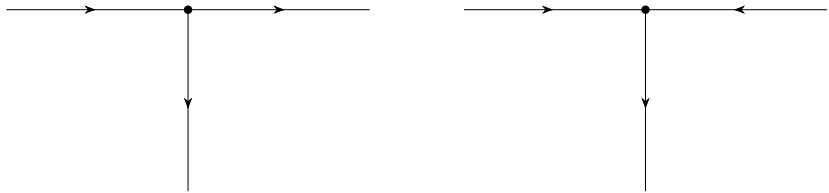

\subsubsection*{Example 1 -- Convergence Tests} 
In the first example, we consider a problem with a smooth density and velocity profile to confirm the convergence rate of the proposed AP method. 
We consider both the 1-to-2 and 2-to-1 T-junction setups. For ingoing pipelines, we take the initial conditions 
$$
u(x,0) = 0,\qquad
\rho(x,0) = \left\{
    \begin{aligned}
        &1.1 && x \leq \frac{0.4}{\eps} \\
        &1+0.1\sin\p{\frac{\pi\eps}{0.8}x}  && \frac{0.4}{\eps} < x < \frac{0.8}{\eps}, \\
        & 1 && x > \frac{0.8}{\eps},
    \end{aligned}\right.
$$
and for outgoing pipes we take the initial conditions $\rho(x,0) = 1$, $u(x,0) = 0$.
We will take each pipe to have a length of $\eps^{-1}$, as we are more interested in the investigating behavior near the junction than the inlet/outlet boundaries.
These pipe lengths allow for the smooth profile to pass through the junction before the final time of $t = 0.2$, but the smooth profile will not reach the inlets/outlets by final time. 
At the inlets, we take a Dirichlet boundary condition of $\rho = 1.1$, and at the outlets we take Neumann boundary conditions. 

Since the exact solution is unknown, we obtain the experimental $L^1$ convergence rates by computing the solution on a number of different meshes, then use the Runge formula 
\begin{equation*}%\label{eq:rates}
    {\rm Rate}_\dx(\rho) = \log_2\p{\frac{\|\rho(x,1)_{\dx}-\rho(x,1)_{2\dx}\|_1}{\|\rho(x,1)_{\dx/2}-\rho(x,1)_\dx\|_1}},
\end{equation*}
where $\rho_\dx$ denotes the density solution computed on a uniform mesh with all cells having width $\dx$.
The convergence rates for $u$ can be computed analogously. 
We present the $L^1$ errors and rates for $\rho$ and $u$ for the 1-to-2 junction in Table \ref{t:1to2_conv} and for the 2-to-1 junction in Table \ref{t:2to1_conv}, both of which confirm the expected first order of accuracy when the fluid passes through the junction. 

\begin{table}[ht!]
    \centering
    \footnotesize{
    \begin{tabular}{l c c c c c}
    \hline
        &$\dx$ & $\|\rho(x,1)_{\dx/2}-\rho(x,1)_\dx\|_1$& Rate$_\dx(\rho)$ & $\|u(x,1)_{\dx/2}-u(x,1)_\dx \|_1$& Rate$_\dx(u)$\\ \hline
$\eps = 0.1$&1/10 &	 1.43e-02  &	 --     &	 1.39e-01 &	 --	\\
&1/20 &	 7.72e-03  &	 0.89 &	 7.57e-02 &	 0.88	\\
&1/40 &	 3.89e-03 &    0.99 &	 3.93e-02 &	 0.94 \\
&1/80 &	 1.97e-03 &	 0.98 &	 2.05e-02 &	 0.94 \\
&1/160 &	 9.85e-04 &	 1.00 &	 1.05e-02 &	 0.96 \\[6pt] \hline  
$\eps = 0.01$&1/10 &	 8.59e-02  &	 --     &	 1.20e+01 &	 --	\\
&1/20 &	 3.08e-02  &	 1.48 &	 3.43e+00 &	 1.81	\\
&1/40 &	 9.53e-03 &    1.69 &	 1.08e+00 &	 1.67 \\
&1/80 &	 2.82e-03 &	 1.76 &	 3.08e-01 &	 1.80 \\
&1/160 &	 9.65e-04 &	 1.55 &	 9.94e-02 &	 1.63 \\ [6pt] \hline 
$\eps = 0.001$&1/10 &	 2.89e-02  &	 --     &	 8.21e+00 &	 --	\\
&1/20 &	 9.10e-03  &	 1.67 &	 1.38e+00 &	 2.58 \\
&1/40 &	 3.06e-03 &    1.57 &	 5.28e-01 &	 1.38 \\
&1/80 &	 1.01e-03 &	 1.60 &	 2.22e-01 &	 1.25 \\
&1/160 &	 3.64e-04 &	 1.47 &	 1.04e-01 &	 1.09 
\\[6pt] \hline
    \end{tabular}}
    \caption{\sf Example 1 (1-to-2 T-junction): $L^1$ errors and corresponding experimental convergence rates.}
    \label{t:1to2_conv}
\end{table}

\begin{table}[ht!]
    \centering
    \footnotesize{
    \begin{tabular}{l c c c c c}
    \hline
        &$\dx$ & $\|\rho(x,1)_{\dx/2}-\rho(x,1)_\dx\|_1$& Rate$_\dx(\rho)$ & $\|u(x,1)_{\dx/2}-u(x,1)_\dx \|_1$& Rate$_\dx(u)$\\ \hline
$\eps = 0.1$&1/10 &	 1.49e-02  &	 --     &	 1.41e-01 &	 --	\\
&1/20 &	 6.98e-03  &	 1.10 &	 8.33e-02 &	 0.76	\\
&1/40 &	 3.35e-03 &    1.06 &	 4.42e-02 &	 0.91 \\
&1/80 &	 1.67e-03 &	 1.01 &	 2.27e-02 &	 0.96 \\
&1/160 & 8.42e-04 &	 0.99 &	 1.14e-02 &	 0.99 \\[6pt] \hline  
$\eps = 0.01$&1/10 &	 9.23e-02  &	 --     &	 1.19e+01 &	 --	\\
&1/20 &	 3.84e-02  &	 1.26 &	 3.01e+00 &	 1.98	\\
&1/40 &	 1.19e-02 &    1.69 &	 9.56e-01 &	 1.66 \\
&1/80 &	 2.99e-03 &	 1.99 &	 3.12e-01 &	 1.61 \\
&1/160 & 9.44e-04 &	 1.66 &	 1.21e-01 &	 1.37 \\ [6pt] \hline 
$\eps = 0.001$&1/10 &	 2.93e-02  &	 --     &	 8.16e+00 &	 --	\\
&1/20 &	 9.33e-03  &	 1.65 &	 1.35e+00 &	 2.60 \\
&1/40 &	 3.18e-03 &    1.55 &	 5.12e-01 &	 1.40 \\
&1/80 &	 1.07e-03 &	 1.57 &	 2.11e-01 &	 1.27 \\
&1/160 & 4.02e-04 &	 1.42 &	 9.78e-02 &	 1.11 
\\[6pt] \hline
    \end{tabular}}
    \caption{\sf Example 1 (2-to-1 T-junction): $L^1$ errors and corresponding experimental convergence rates.}
    \label{t:2to1_conv}
\end{table}

\subsubsection*{Example 2 -- Inlet Discontinuity}

For the second example, we look at the evolution of an initial jump discontinuity that starts at the inlets, resulting in a wave that travels through the junction. 
For all pipelines, we consider the initial conditions 
$$
\rho(x,0) = 1,\qquad u(x,0) = 0,
$$
with each pipe having a length of 100. 
At the inlets, we take a Dirichlet boundary condition of $\rho = 1.3$, and prescribe Neumann boundary conditions at the outlets. 
We look at this initial set up for both the 1-to-2 and 2-to-1 T-junction networks. 

We use the proposed AP scheme to compute the solution for $\eps$ values of $0.1$, $0.01$, and $0.001$ to the final times of $t = 10$, $t = 1$, and $t = 0.1$, respectively, to capture the behavior of the discontinuity through the junction. 
The resulting solutions for $\rho$ and $u$ when taking a spatial mesh of 4000 cells in each pipeline are presented in Figure \ref{fig:disc_1to2} for the 1-to-2 T-junction and in Figure \ref{fig:disc_2to1} for the 2-to-1 T-junction. 
In these figures, the influence of the friction source term is very clear, as it appears as if the flow did not yet reach the junction in the $\eps = 0.001$ results. 
This is however not the case and is only a result of the strong friction; see Figure \ref{fig:eps001_outs} where we plot the outgoing pipeline results for $\eps = 0.001$ at $t = 0.1$.
In addition, we see that in the cases that produce discontinuities in the solution, there are no spurious oscillations present.

\begin{figure}[ht!]
    \centering
    \includegraphics[trim={1.5cm 2.0cm 0.5cm 2.0cm},clip,width=0.4\textwidth]{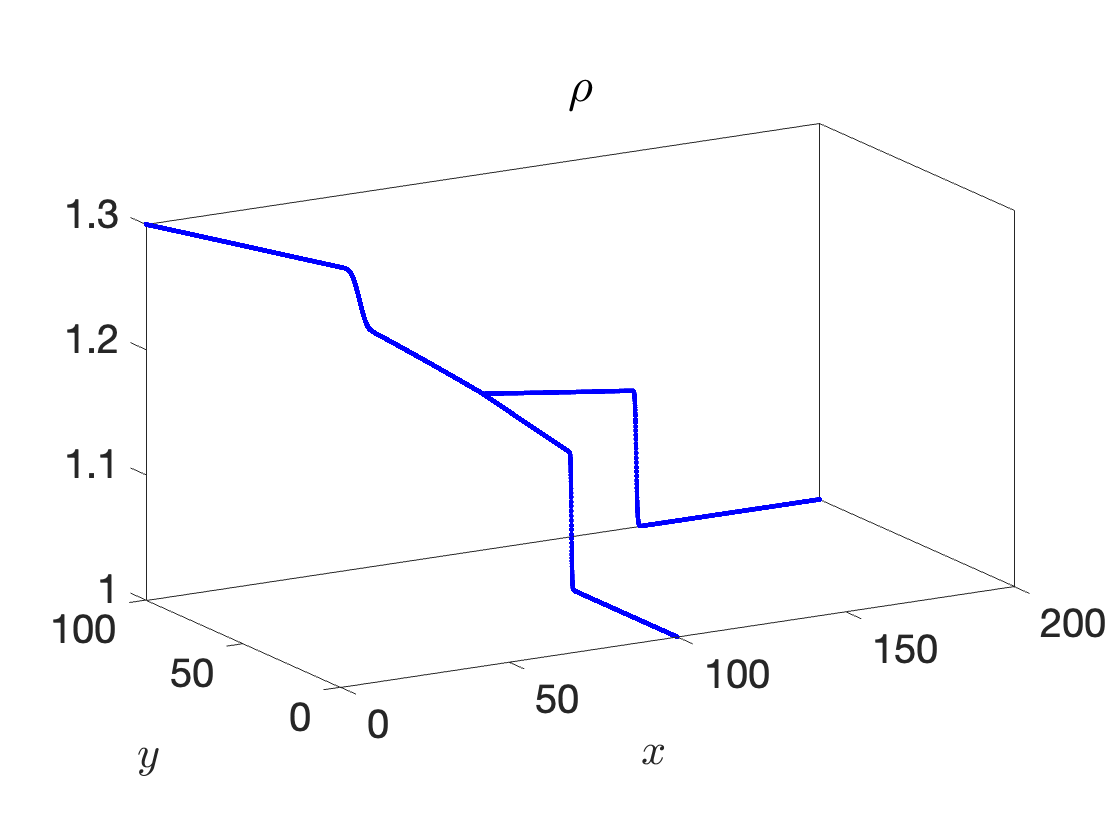}
    \hspace{1cm}
    \includegraphics[trim={1.5cm 2.0cm 0.5cm 2.0cm},clip,width=0.4\textwidth]{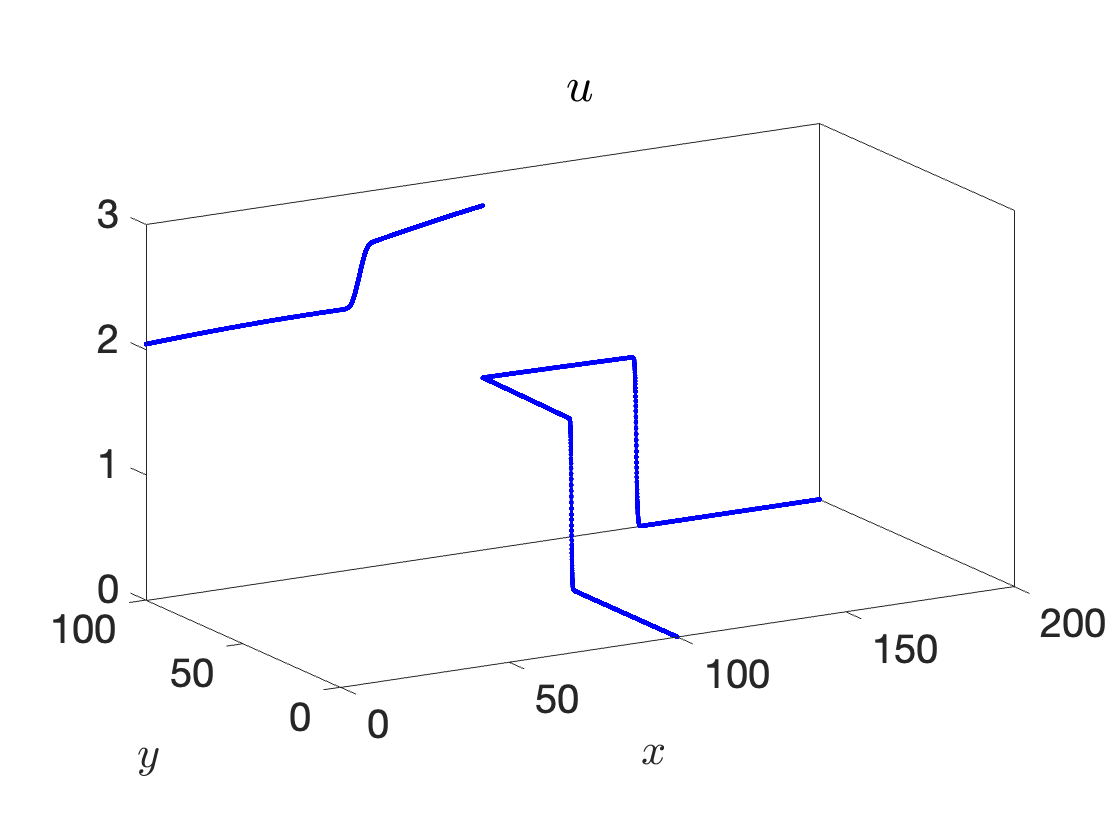}

    \includegraphics[trim={1.5cm 2.0cm 0.5cm 2.0cm},clip,width=0.4\textwidth]{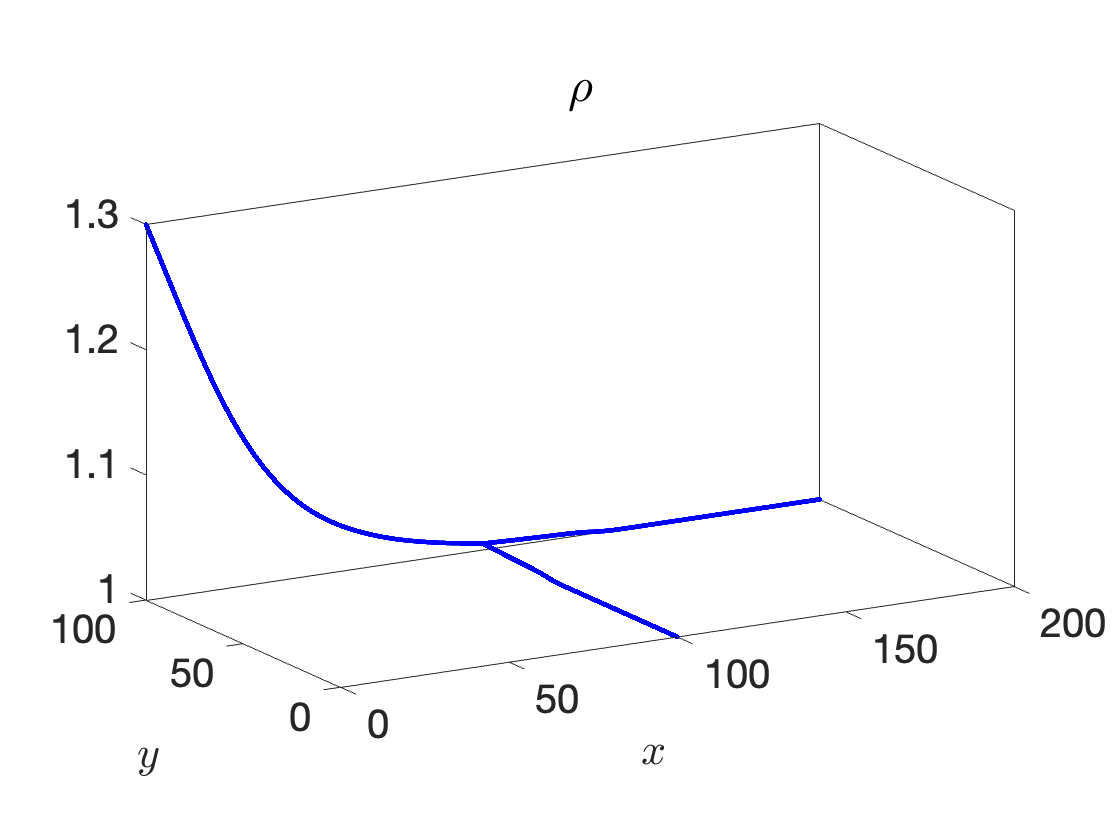}
    \hspace{1cm}
    \includegraphics[trim={1.5cm 2.0cm 0.5cm 2.0cm},clip,width=0.4\textwidth]{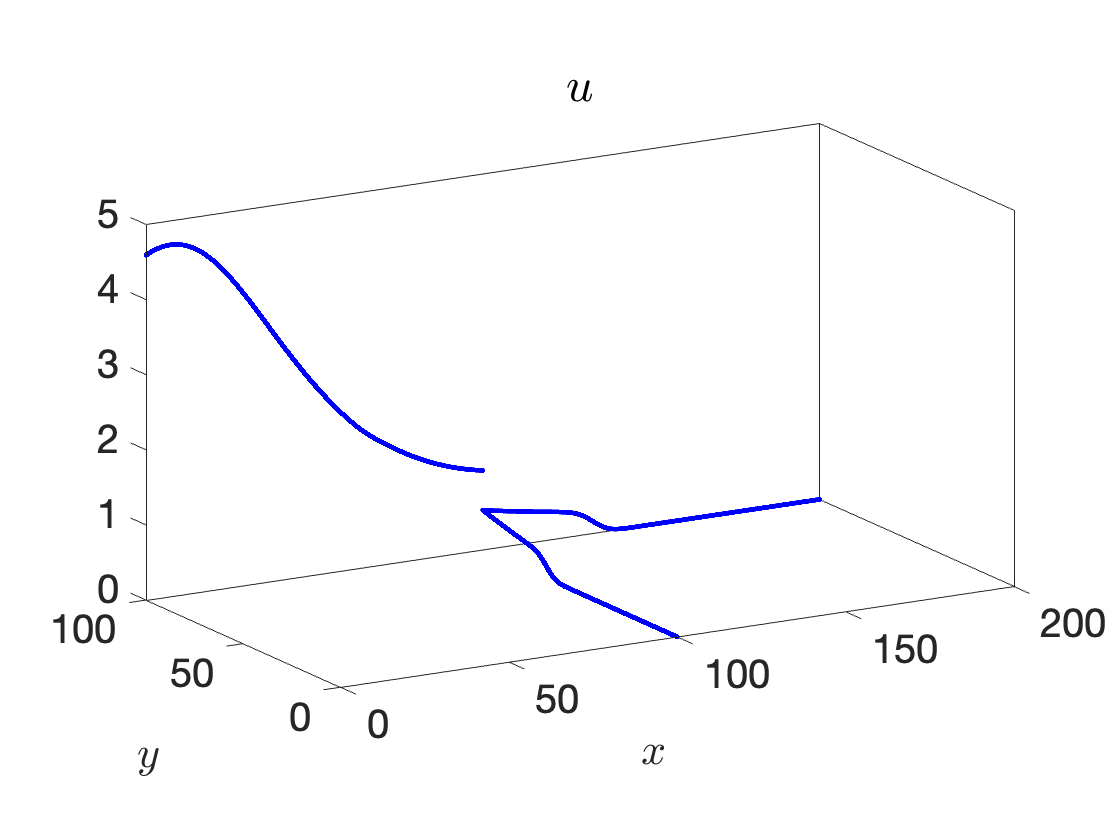}

    \includegraphics[trim={1.5cm 2.0cm 0.5cm 2.0cm},clip,width=0.4\textwidth]{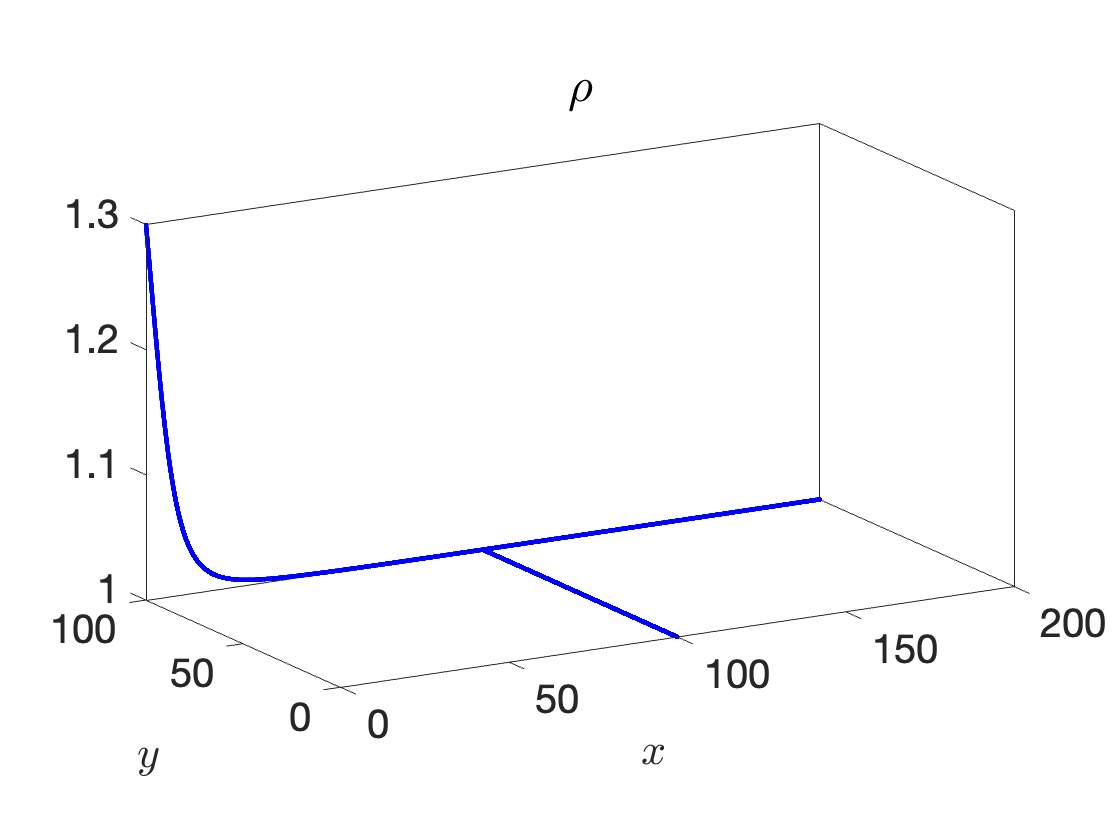}
    \hspace{1cm}
    \includegraphics[trim={1.5cm 2.0cm 0.5cm 2.0cm},clip,width=0.4\textwidth]{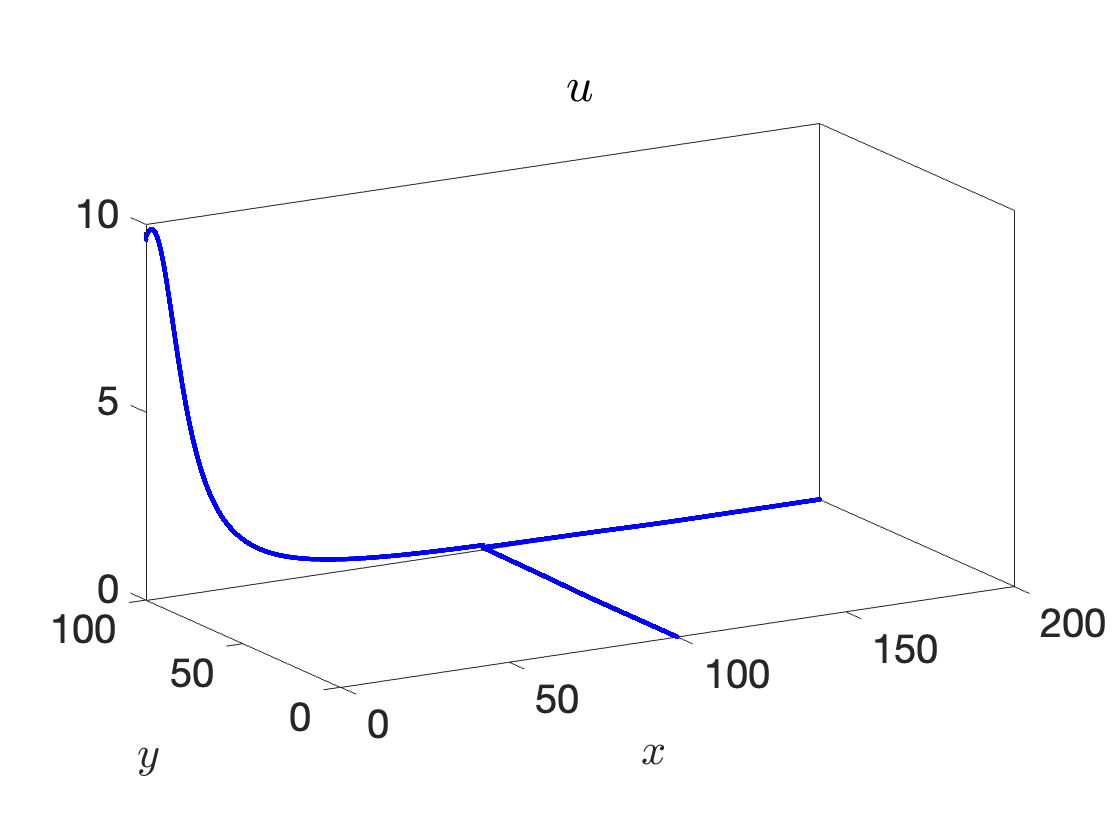}
    \caption{\sf Example 2 (1-to-2 T-junction): Solutions for $\rho$ (left column) and $u$ (right column) for $\eps = 0.1$ at $t = 10$ (top row), $\eps = 0.01$ at $t = 1$ (middle row), and $\eps = 0.001$ at $t = 0.1$ (bottom row).}
    \label{fig:disc_1to2}
\end{figure}

\begin{figure}[ht!]
    \centering
    \includegraphics[trim={1.5cm 2.0cm 0.5cm 2.0cm},clip, width=0.4\textwidth]{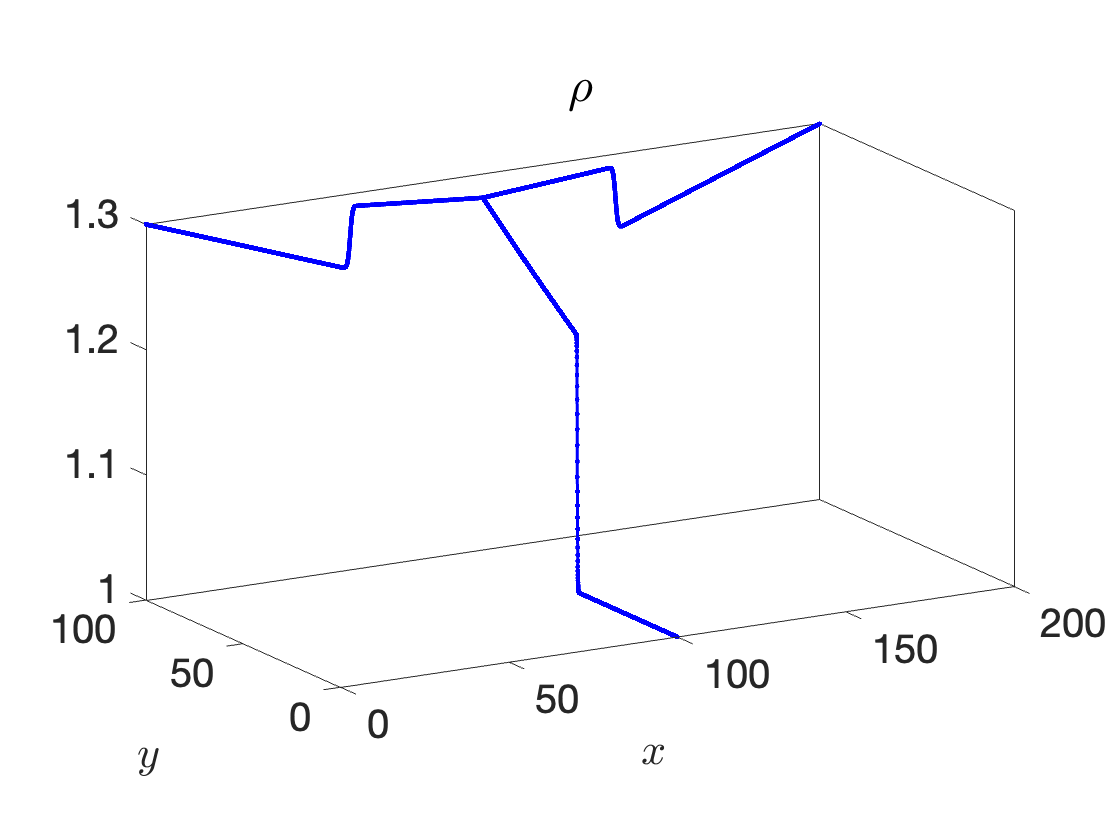}
    \hspace{1cm}
    \includegraphics[trim={1.5cm 2.0cm 0.5cm 2.0cm},clip,width=0.4\textwidth]{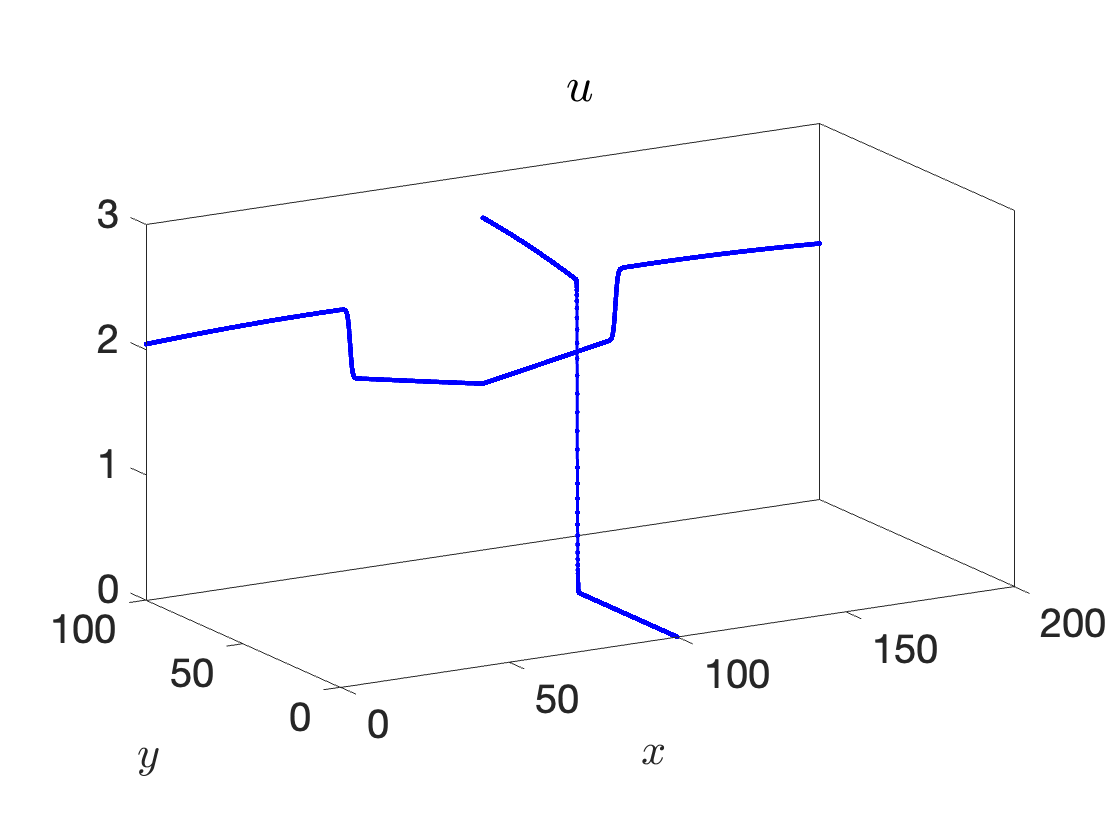}

    \includegraphics[trim={1.5cm 2.0cm 0.5cm 2.0cm},clip,width=0.4\textwidth]{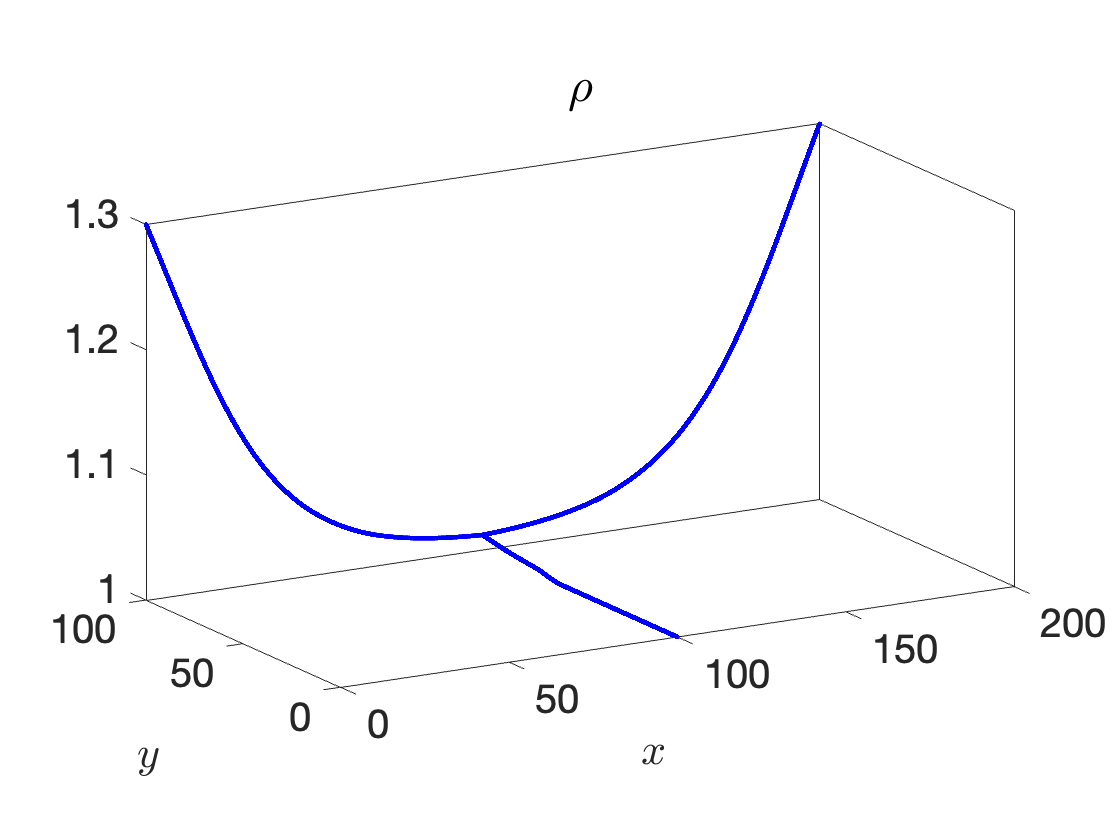}
    \hspace{1cm}
    \includegraphics[trim={1.5cm 2.0cm 0.5cm 2.0cm},clip,width=0.4\textwidth]{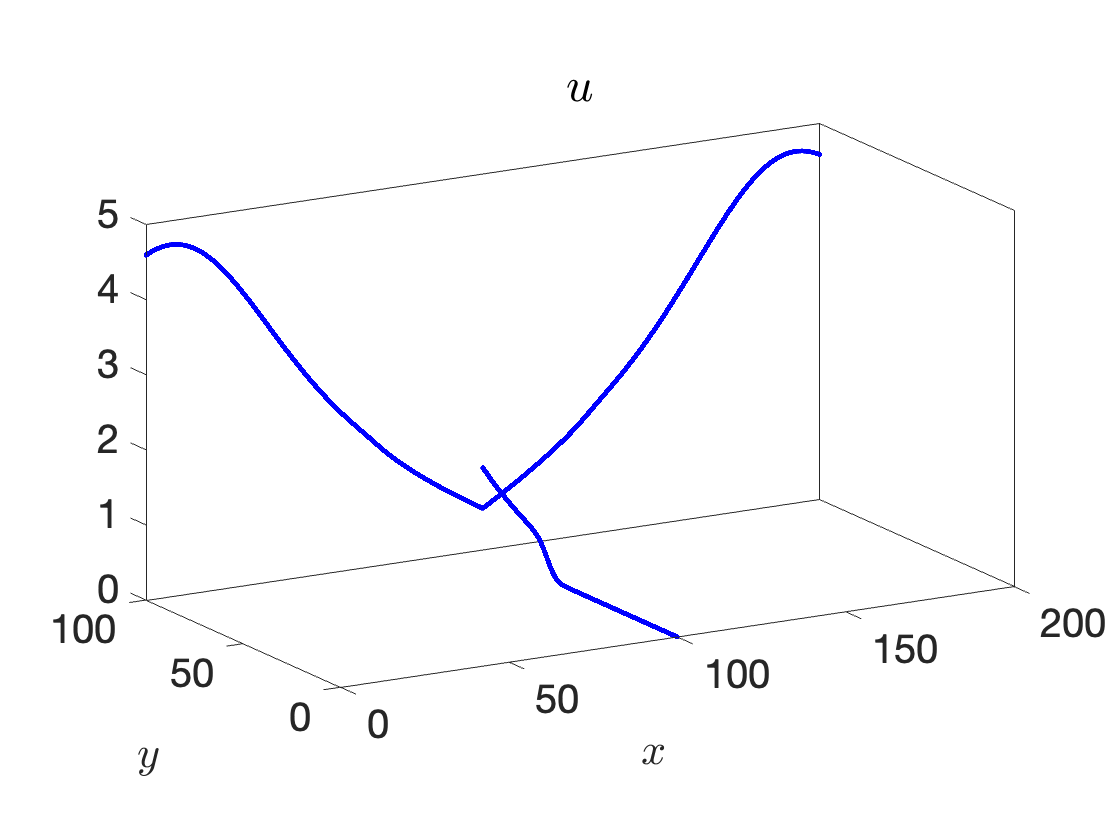}

    \includegraphics[trim={1.5cm 2.0cm 0.5cm 2.0cm},clip,width=0.4\textwidth]{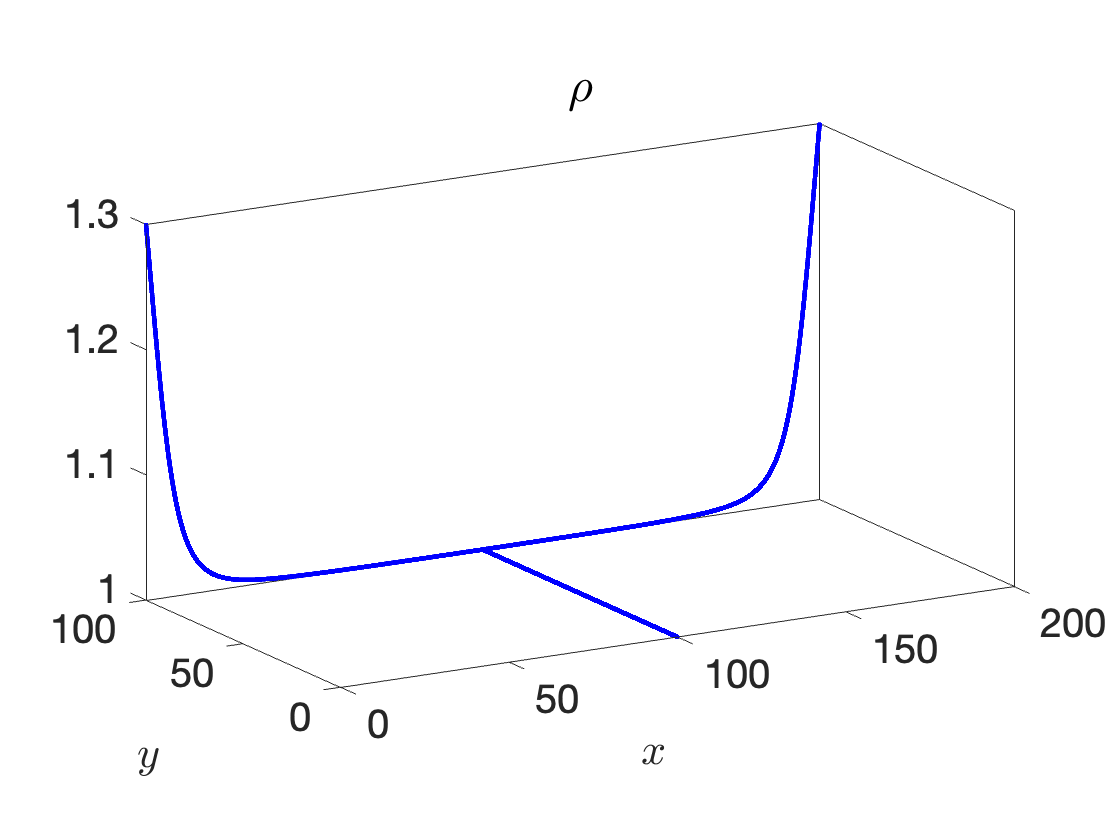}
    \hspace{1cm}
    \includegraphics[trim={1.5cm 2.0cm 0.5cm 2.0cm},clip,width=0.4\textwidth]{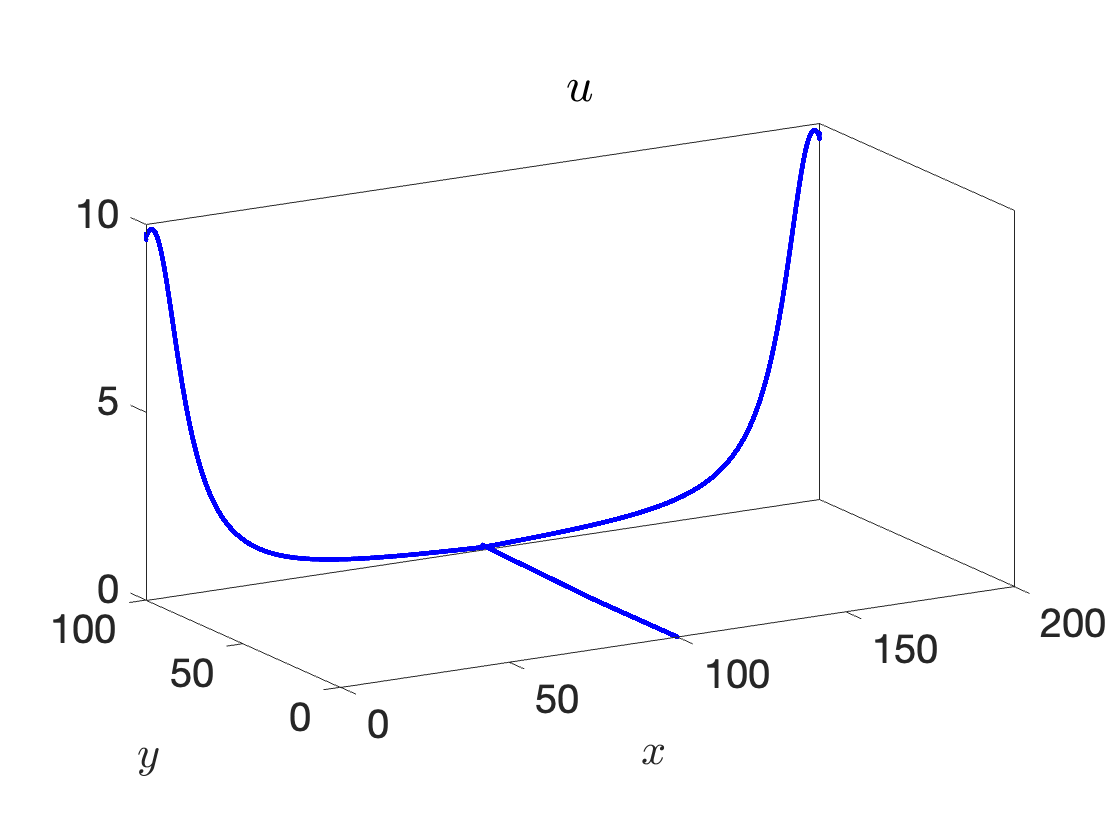}
    \caption{\sf Example 2 (2-to-1 T-junction): Solutions for $\rho$ (left column) and $u$ (right column) for $\eps = 0.1$ at $t = 10$ (top row), $\eps = 0.01$ at $t = 1$ (middle row), and $\eps = 0.001$ at $t = 0.1$ (bottom row).}
    \label{fig:disc_2to1}
\end{figure}

\begin{figure}[ht!]
    \centering
    \includegraphics[width=0.4\textwidth]{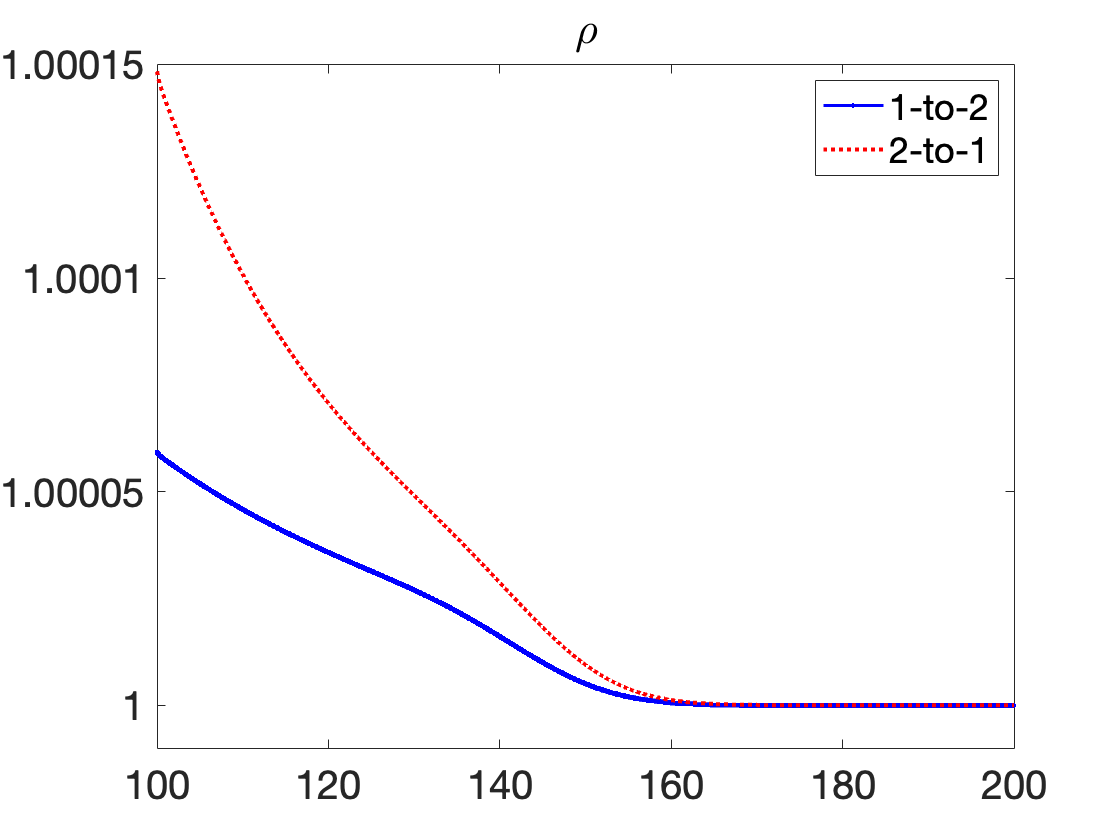}
    \hspace{1cm}
    \includegraphics[width=0.4\textwidth]{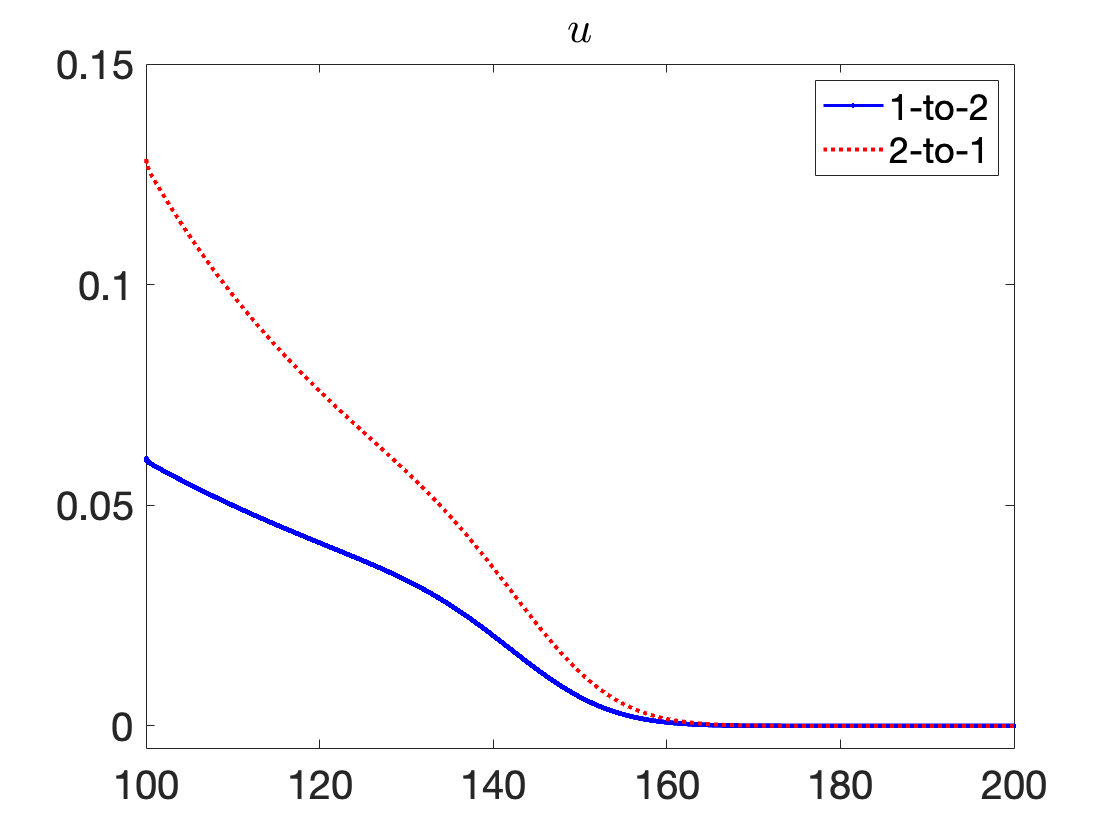}
    \caption{\sf Example 2: Results for $\rho$ (left) and $u$ (right) of the outgoing pipeline(s) for the 1-to-2 T-junction (blue) and the 2-to-1 T-junction (red dotted).}
    \label{fig:eps001_outs}
\end{figure}

\subsubsection*{Example 3 -- Comparison with the CU Explicit Scheme}
In this example, we test the importance of the $\eps$-independent time-step restriction and compare the proposed AP scheme against the CU finite volume discretization with a forward Euler time-step. 
The forward Euler time-step was selected so that the explicit scheme matches the first-order in time, second-order in space accuracy of the proposed AP scheme.

We consider the 1-to-2 T-junction with the same initial and boundary conditions as that of Example 2, and run both the proposed AP scheme and the related explicit scheme to a final time $t = 10$ for the $\eps$ values of 0.1, 0.01, and 0.001 on various spatial discretizations. 
We present the run times of both schemes in Table \ref{t:run_times}. 
In the run times, we see that for $\eps = 0.1$, the run times for the two methods are on the same order of magnitude. 
However, the difference becomes drastic as $\eps \rightarrow 0$.
It is clear in Table \ref{t:run_times} that the AP scheme keeps all run times roughly the same for varying $\eps$.
Conversely, the explicit scheme has the expected $\eps$-dependence within the experimental run times -- which, for $\eps = 0.001$, results in simulations that are a daunting 140$ \times$ slower than that of the AP scheme.

\begin{table}[ht!]
\centering
\footnotesize{
\begin{tabular}{l c c c}
\hline
&$\dx$ & AP Scheme Run Time& Explicit Scheme Run Time \\ \hline
$\eps = 0.1$
&1/20 &	 2.44 s &	 4.05 s\\
&1/40 &	 9.14 s &    16.1 s \\
&1/80 &	 37.4 s &	 65.5 s \\[6pt] \hline  
$\eps = 0.01$
&1/20 &	 2.89 s &	 46.8 s \\
&1/40 &	 11.0 s &    192 s \\
&1/80 &	 44.0 s &	 782 s \\[6pt] \hline 
$\eps = 0.001$
&1/20 &	 2.91 s &	 406 s \\
&1/40 &	 11.0 s &    1610 s \\
&1/80 &	 43.7 s &	 6450 s \\[6pt] \hline
    \end{tabular}}
    \caption{\sf Example 3: Run time comparisons of the proposed AP scheme and CU discretization with forward Euler time-stepping for various $\eps$ and mesh sizes. }
    \label{t:run_times}
\end{table}

%================================
\section{Conclusion}\label{sec5}

In this study, we developed a novel asymptotic-preserving method for the isentropic Euler equations with a friction source on pipe networks. 
To address the difficulties and stiffness of the system in low Mach or high friction regimes, we split the flux into pieces that represent the slow and fast dynamics.  This allows for an explicit time-step on the non-stiff (slow) dynamics portion, and, by using a method based on Rosenbrock-type Runge Kutta schemes, allows for an implicit time-step for the stiff (fast) dynamics that does not require a nonlinear solver. 
In turn, we were able to confirm both experimentally and theoretically that the proposed scheme is AP in the sense that it provides a consistent and stable solution in the the low Mach and high friction regimes. 
Most importantly, the CFL stability restriction is only dependent on mesh size, rather than requiring dependence on both the mesh and the small limiting parameter $\eps$ within the system.
This method within each pipeline is then extended to entire pipe networks, in which coupling conditions must be used at pipe-to-pipe intersections to ensure a mathematically well-posed problem. 
We show that, even in the limiting regime, the coupling conditions remain well-posed. 
These coupling conditions are used to set the ghost cell values on each pipeline, thus allowing a seamless coupling of the AP method across pipe junctions within the network. 

Since the AP method is combined with a friction source term and pipe junctions, both the asymptotic state in the low Mach/high friction regime and the boundary conditions are non-trivial. 
Because of this, in addition to a first-order approximation in time, the scheme is currently limited to first-order. 
The second-order extension to the proposed method is, to our knowledge, non-trivial, as it would require, e.g., one-sided reconstructions, adjustments to the half-Riemann problems, and approximating fluxes in the ghost cells. 
In future work, these difficulties are something we plan to address in hopes of making a fully second-order AP scheme for isentropic flow in pipe networks.

%================================
\section*{Acknowledgments}

The authors thank the Deutsche Forschungsgemeinschaft (DFG, German Research Foundation) for the financial support through
463312734/FOR5409, 
320021702/GRK2326, and 
within SPP 2410 Hyperbolic Balance Laws in Fluid Mechanics: Complexity, Scales, Randomness (CoScaRa) within the Project(s) HE5386/26-1 (Numerische Verfahren für gekoppelte Mehrskalenprobleme, 525842915) and 
(Zufällige kompressible Euler Gleichungen: Numerik und ihre Analysis, 
525853336) HE5386/27-1. Support received funding from the European Union’s Horizon Europe research and innovation programme under the Marie Sklodowska-Curie Doctoral Network Datahyking (Grant No. 101072546) is acknowledged.

%\newpage
\bibliographystyle{siam}
\bibliography{biblio,references}
\end{document}